\documentclass[11pt]{amsart}
\usepackage{amsmath,amssymb,amsfonts}
\usepackage{verbatim}

\usepackage{times}
\usepackage{amssymb}
\usepackage{amsmath}

\makeatletter
\@namedef{subjclassname@2010}{%
  \textup{2010} Mathematics Subject Classification}
\makeatother

\newtheorem{thm}{Theorem}[section]
\newtheorem{prop}[thm]{Proposition}

\newtheorem{lemma}[thm]{Lemma}

\newtheorem{cor}[thm]{Corollary}

\theoremstyle{plain}
\numberwithin{equation}{section}

\theoremstyle{remark}
\newtheorem{remark}[thm]{\bf Remark}

\newtheorem{remarks}[thm]{\bf Remarks}

\newcommand{\Comment}[1]{}
\newcommand{\rbr}[1]{\left( {#1} \right)}

\newcommand{\cbr}[1]{\left\{ {#1} \right\}}
\newcommand{\abs}[1]{\left| {#1} \right|}

\newcommand{\eq}[1]{(\ref{#1})}

\def\divides{\mid}
\def\ndivides{\hspace{-0.25 em} \not \hspace{.25 em} \mid}

\title{A variant of the Bombieri-Vinogradov theorem with explicit constants and applications}

\author{Amir Akbary and Kyle Hambrook}

\thanks{Research of the authors is partially supported by NSERC and ULRF}



\address{Department of Mathematics and Computer Science \\
        University of Lethbridge \\
        Lethbridge, AB T1K 3M4 \\
        Canada}
\email{amir.akbary@uleth.ca}

\address{Department of Mathematics\\
1984 Mathematics Road\\
University of British Columbia\\
Vancouver, BC V6T 1Z2\\
Canada}
\email{hambrook@math.ubc.ca}

\begin{document}

\begin{abstract}
We give an effective version with explicit constants of a mean value theorem of Vaughan related to
the values of $\psi(y, \chi)$, the twisted summatory function associated to the von Mangoldt function $\Lambda$ and a Dirichlet character $\chi$.
As a consequence of this
result we prove an effective variant of the Bombieri-Vinogradov theorem with explicit constants.  
This effective variant has the potential to provide explicit results in many problems. We give examples of
such results in several number theoretical problems related to shifted primes.
\end{abstract}
\maketitle

\section{Introduction}
For integers $a$ and $q \geq 1$, let 
$$
\psi(x; q, a)=\sum_{\substack{n\leq x \\ {n\equiv a\bmod{q}}}} \Lambda(n),
$$
where $\Lambda(n)$ is the von Mangoldt function. 
If $q=1$, $\psi(x;1,a)=\psi(x):=\sum_{\substack{n\leq x }} \Lambda(n)$.
The prime number theorem in arithmetic progressions
is the assertion that
$$ \psi(x; q, a)\sim \frac{x}{\phi(q)},$$
as $x\rightarrow \infty$, where $\phi(q)$ is the Euler function. The Bombieri-Vinogradov theorem is an estimate for the error terms in the prime number theorem for arithmetic progressions averaged over all moduli $q$ up to almost $x^{1/2}$. 
The version given in 1965 by Bombieri \cite{B} states that if $A$ is a given positive number and $Q\leq x^{1/2} (\log{x})^{-B}$ where
$B=3A+23$, then
$$\sum_{q \leq Q} \max_{2 \leq y \leq x} \mathop{\max_{a}}_{(a,q)=1} \abs{ \psi(y,q,a) - \frac{y}{\phi(q)} } \ll_A \frac{x}{(\log{x})^A}.$$
Over the years the value of $B$ has been improved. For instance, Dress, Iwaniec, and Tenenbaum \cite{DIT} proved $B=A+5/2$ is valid. The implied constant in the Bombieri-Vinogradov theorem is not effective. This is ultimately due to the need to allow for the possible existence of exceptional characters associated to those $q$ in the sum having small prime factors.

The Bombieri-Vinogradov theorem has been used to prove a number of interesting results. One such result is the following estimate on the number of primes whose shifts have large prime factors. For the statement, we introduce the notation $P(n)$ for the greatest prime divisor of an integer $n\neq 0, \pm 1$. Also, throughout the paper $p$ always denotes a prime number.
\begin{thm}\label{Theorem 1.1 of Harman}
There is a $\theta_0 > 1/2$ such that for all non-zero $a \in \mathbb{Z}$ and all $\theta \leq \theta_0$ there are positive constants $X_0(a,\theta)$ and $\delta(\theta)$ for which
$$
\sum_{\substack{p \leq x \\ P(p+a) > x^{\theta}}} 1 > \delta(\theta) \frac{x}{\log x}
$$
whenever $x \geq X_0(a,\theta)$.
\end{thm}
This was first proved by Goldfeld \cite{G} with $\theta_0 = 0.583\ldots$. The best version to date was proved by Baker and Harman \cite{BH}, who obtained $\theta_0=0.677$.
The proofs of Goldfeld and of Baker and Harman give no way of calculating $X_0(a,\theta)$
as they rely on the Bombieri-Vinogradov theorem and other ineffective results. In \cite{H}, Harman showed how to use an effective variant of the Bombieri-Vinogradov theorem (in which the moduli $q$ are restricted to the primes) to prove Theorem \ref{Theorem 1.1 of Harman} with 
$\theta_0 = 0.6105\ldots$ 
and effectively computable $\delta(\theta)$ and $X_0(a, \theta)$. While Harman asserted that calculating $X_0(a,\theta)$ is feasible, he did not provide an explicit value for $X_0(a, \theta)$ for any values of $a$ and $\theta$.

Inspired by Harman's work, in this paper we provide an effective variant of the Bombieri-Vinogradov theorem with explicit constants where $q$ varies over integers. We then apply our variant to several number theoretical problems involving shifted primes to obtain new effective results. One of our results (Corollary  \ref{Corollary 1.10} (ii)) provides  numerical values for the constants $\delta(\theta)$ and $X_0(a, \theta)$ in Harman's theorem (see \cite[Theorem 2]{H}) for $a=-1$ and $\theta=0.6$.

We now describe our effective variant in detail.


In 1977, following the ideas of I. M. Vinogradov, Vaughan \cite{V4} described a method for estimating sums of the form
$$\sum_{p\leq x} f(p),$$
where $f(n)$ is either an exponential function of the form $e^{2\pi i \alpha n}$ or a Dirichlet character $\chi(n)$ modulo $q$.
He then applied this method to give an elementary proof of the Bombieri-Vinogradov theorem \cite{V5}.
The main step in Vaughan's proof is a mean value theorem for 
$$\psi(y,\chi)= \sum_{n\leq y} \Lambda(n) \chi(n),$$
the twisted summatory function associated to the von Mangoldt function $\Lambda$ and a Dirichlet character
$\chi$. Our first theorem provides a version of Vaughan's mean value inequality with explicit constants.

\begin{thm}
\label{TMVT}
For $x \geq 4$,
\begin{align}
\label{MVT}
\sum_{q \leq Q} \frac{q}{\phi(q)} \mathop{{\sum}^{\ast}}_{\chi \hspace{0.15 em} (q)} \max_{y \leq x} \abs{ \psi(y,\chi) } <  c_0 \rbr{ 4x + 2x^{\frac{1}{2}}Q^2 + 6x^{\frac{2}{3}} Q^{\frac{3}{2}}+ 5x^{\frac{5}{6}}Q } (\log x)^{\frac{7}{2}},
\end{align}
where
\begin{align*}
c_0 
&= \frac{2^{\frac{13}{2}}}{9\pi \log{2}}\rbr{\frac{1}{3} + \frac{3}{2\log2}} \left( \frac{2 + \log( \log(2) /\log(4/3))}{\log 2} \right)A_0^{\frac{1}{2}} \\
&= 48.83236\ldots.
\end{align*}
Here
$$
A_0 = \max_{x >0} \rbr{\frac{\psi(x)}{x}} = \frac{\psi(113)}{113} = 1.03883\ldots,
$$
$Q$ is any positive real number, \Comment{$E_0= \prod_{p} \rbr{1 + \frac{1}{p(p-1)}}$, }and $\mathop{{\sum}^{\ast}}_{\hspace{-0.4 em}\chi \hspace{0.15 em} (q)}$ denotes a sum over all primitive characters $\chi$ (mod~$q$).
\end{thm}

Mean value theorems of this type have many applications (see \cite{V1} and \cite{V2} for some examples). One of the most important consequences of \eqref{MVT} is a straightforward proof of 
the Bombieri-Vinogradov theorem (see \cite[Chapter 28]{D} or \cite{V5}). In addition to Theorem \ref{TMVT}, the proof uses the Siegel-Walfisz theorem, which states that
$$\psi(x,\chi)-\delta(\chi) x \ll_A x \exp\left(-c \sqrt{\log{x}} \right)$$
uniformly for $q\leq (\log{x})^A$. Here $A>0$ is a fixed real number, $c$ is an absolute positive constant, and $\delta(\chi)$ is $1$ or $0$ according to whether the Dirichlet character $\chi$ is principal or non-principal. The implied constant in the Bombieri-Vinogradov theorem is ineffective because the implied constant in the Siegel-Walfisz theorem is ineffective. 
In the proof of the Bombieri-Vinogradov theorem, one uses the Siegel-Walfisz theorem to deal with moduli $q \leq Q$ having small prime divisors and uses Theorem \ref{TMVT} to deal with the sum over the remaining $q \leq Q$.
The following theorem is simply what one gets for the sum over these remaining $q$; we refer to this theorem as an effective variant of the Bombieri-Vinogradov theorem.

\begin{thm} \label{EBVT}
Let $x \geq 4$ and $1 \leq Q_1 \leq Q \leq x^{\frac{1}{2}}$.  Let $\ell(q)$ denote the least prime divisor of $q$.  Then 
\begin{align} \label{EBVTpsi}
\sum_{\substack{q\leq Q \\ \ell(q)>Q_1}} \max_{2 \leq y \leq x} \mathop{\max_{a}}_{(a,q)=1} \abs{ \psi(y;q,a) - \frac{\psi(y)}{\phi(q)} } 
< c_1 \rbr{ 4\frac{x}{Q_1} + 4x^{\frac{1}{2}}Q + 18x^{\frac{2}{3}}Q^{\frac{1}{2}} + 5x^{\frac{5}{6}} \log \rbr{ \frac{eQ}{Q_1} }} \rbr{\log x}^{\frac{9}{2}}.
\end{align}
Here $e=\exp(1)$ and $c_1= (5/4)E_0c_0+1$, where $E_0 = \prod_{p} \rbr{1 + \frac{1}{p(p-1)}} = 1.943596\ldots$.
\end{thm}


Effective variants of the Bombieri-Vinogradov theorem have been known for a while. However, as far as we know, prior to our work there has not been a version with  explicit numerical constants. Timofeev \cite[Theorem 2]{T} proved an effective variant 
(without
constants explicitly given) when the modulus $q$ varies over integers with $\ell(q)>\exp{((\log{x})^{1/4})}$.
Our effective variant is stronger and produces upper bounds in the form $x/(\log{x})^A$ for $\ell(q)>\exp{((\frac{9}{2}+A)\log\log{x})}$. 
{Another effective variant 
without explicit constants is given by Lenstra and Pomerance \cite[Lemma 11.2] {LP} in their work on Gaussian periods and a polynomial time primality testing algorithm.}

Next we state the version of Theorem \ref{EBVT} that we will employ in the applications.  
As usual, let
\begin{align*}
\pi(x) = \sum_{p \leq x} 1 \qquad \text{ and } \qquad \pi(x;q,a) = \sum_{\substack{p \leq x \\ p \equiv a \bmod{q} }} 1.
\end{align*}

\begin{cor}
\label{CEBVT}
Under the assumptions of Theorem \ref{EBVT},
\begin{align} \label{EBVTpi}
\sum_{\substack{q \leq Q \\ \ell(q) > Q_1}} \max_{2 \leq y \leq x} \mathop{\max_{a}}_{(a,q)=1} \abs{ \pi(y;q,a) - \frac{\pi(y)}{\phi(q)} } 
< c_2 \rbr{ 4\frac{x}{Q_1} + 4x^{\frac{1}{2}}Q + 18x^{\frac{2}{3}}Q^{\frac{1}{2}} +  5x^{\frac{5}{6}} \log \rbr{ \frac{eQ}{Q_1} } } \rbr{\log x}^{\frac{9}{2}},
\end{align}
where 
$c_2 = {2 c_1}/{ \log 2 } + 1$.
\end{cor}

Harman \cite[Theorem 3]{H} gives a version of Corollary \ref{CEBVT} for prime moduli $p \leq Q$ without explicit constants. 
In many applications one needs the version given in Corollary \ref{CEBVT}. Indeed, for Theorem \ref{Erdos} below we have to work with composite modulus $q$.

\begin{remark}
The right-hand side of the inequality \eq{MVT} in Theorem \ref{TMVT} has the form
$$c _0 f(x, Q) (\log{x})^\alpha,$$
where $c_0$ and $\alpha$ are positive real numbers. Our goal was to generate an explicit expression for $f(x, Q)$ which is not far from optimal and at the same time can be neatly written. Following the methods of this paper, there are several ways that one can replace $f(x, Q)$ with a more optimal function. 
As an example, the coefficients of $f(x,Q)$ in \eq{MVT} have been rounded up to be integers. The proof of Theorem \ref{TMVT} can be modified to establish this inequality with slightly smaller non-integral coefficients.
As a more significant example, an application of an improved version of the Polya-Vinogradov inequality (see \cite[Theorem 1]{P}) in \eqref{PV} and \eqref{S3prime} will reduce the values of the coefficients of $f(x, Q)$. Our method also permits further reduction of the exponent $\alpha$.
From \cite{DIT} we know that 
\begin{equation}
\label{optimal}
\sum_{k\leq x} \left(\sum_{ \substack{ d \mid k \\ d \leq V }  } \mu(d) \right)^2\ll x,
\end{equation}
where $\mu$ is the M\"{o}bius function. Employing an explicit version of this inequality (whenever one becomes available) instead of \eqref{norm2}, will result in $\alpha={5}/{2}$.  We note that by the main result of \cite{DIT} the upper bound $x$ in \eqref{optimal} is optimal. Thus it will not be possible to improve $\alpha$ beyond $5/2$ by sharpening the inequality \eqref{optimal}. Similar comments are applicable to the inequalities in Theorem \ref{EBVT} and Corollary \ref{CEBVT}.
\end{remark}



Here we describe the structure of this paper. 
Section 2 presents several applications of our variant of the Bombieri-Vinogradov theorem (Corollary \ref{CEBVT}) in problems related to shifted primes.
In Section 3 we collect some known number theoretic inequalities with explicit constants that will be used in our proofs. 
Sections 4 and 5 are devoted to proofs of the results stated in Section 2.
Section 6 is dedicated to a detailed proof of our main mean value theorem (Theorem \ref{TMVT}).
The proofs of Theorem \ref{EBVT} and Corollary \ref{CEBVT} are given in Sections 7 and 8, respectively.

\medskip\par
\noindent{\bf Notation.}
The symbols $p$, $p_1$, $p_2$, $\ell$, $\ell_1$, and $\ell_2$ always denote primes. The symbols $n$ and $q$ always denote integers. The least prime divisor of $q$ is denoted $\ell(q)$. The number of distinct prime divisors of $n$ is denoted $\omega(n)$. The M\"{o}bius function and the von Mangoldt function are defined, respectively, by
$$
\mu(n) = \left\{\begin{array}{ll}
(-1)^{\omega(n)} & \text{if $n$ is square-free,} \\
0 & \text{otherwise,}
\end{array}
\right.
$$
and
$$
\Lambda(n) = \left\{\begin{array}{ll}
\log p & \text{if $n = p^k$ for some prime $p$ and integer $k \geq 1$,} \\
0 & \text{otherwise.}
\end{array}
\right.
$$
The prime counting function and the standard Chebyshev functions are
$$
\pi(x) = \sum_{p \leq x} 1, \qquad \vartheta(x) = \sum_{p \leq x} \log p, \qquad \psi(x) = \sum_{n \leq x} \Lambda(n).
$$
The corresponding functions in an arithmetic progression are 
$$
\pi(x;q,a) = \sum_{\substack{p\leq x \\ {p\equiv a\bmod{q}}}} 1, \qquad \vartheta(x;q,a) = \sum_{\substack{p\leq x \\ {p\equiv a\bmod{q}}}} \log p, \qquad \psi(x;q,a) = \sum_{\substack{n\leq x \\ {n\equiv a\bmod{q}}}} \Lambda(n),
$$
where $a, q$ are integers with $q \geq 1$. The twisted summatory function associated to the von Mangoldt function $\Lambda$ and a Dirichlet character
$\chi$ is defined by 
$$
\psi(x,\chi) = \sum_{n \leq x} \Lambda(n) \chi(n).
$$
The greatest prime divisor of an integer $n \neq 0,\pm 1$ is denoted $P(n)$. The least prime congruent to $a$ modulo $p$ is denoted $L(p,a)$. The multiplicative order of a nonzero integer $b$ modulo $p$ is denoted $e_b(p)$. The constants $c_0$ and $A_0$ are defined in Theorem \ref{TMVT}. The constants $c_1$ and $E_0$ are defined in Theorem \ref{EBVT}. The constant $c_2$ is defined in Corollary \ref{CEBVT}.
\medskip\par

\noindent {\bf Acknowledgements.}
We thank Adam Felix, Andrew Granville, Glyn Harman, Kumar Murty, Carl Pomerance, Olivier Ramar\'{e}, and Igor Shparlinski for their correspondence and their comments on an earlier draft of this paper. We also thank the referee for many helpful comments and suggestions.

\section{Summary of Applications}
\label{sectiontwo}

In this section we describe some applications of Corollary \ref{CEBVT}. Moreover, we calculate explicitly the constants involved. As far as we know, this has not been done before.

The following two theorems are explicit Tur\'{a}n-type inequalities for $\omega(p-1)$,  
the number of prime divisors of $p-1$.  
\begin{thm}
\label{Erdos}
Given $\epsilon>0$, there is an effective constant $C_0(\epsilon)$ such that
$$\frac{1}{\pi(x)}\sum_{p\leq x} (\omega(p-1)-\log\log{x})^2 < (4+\epsilon) (\log\log{x}) (\log\log\log{x})$$ 
for $x\geq C_0(\epsilon)$.
\end{thm}

\begin{thm}
\label{Erdos1}
For each $b > 0$ there is an effective constant $C(b)$ such that 
$$\frac{1}{\pi(x)}\sum_{p\leq x} (\omega(p-1)-\log\log{x})^2 < \rbr{\frac{251}{60} + \frac{b}{\log \log \log x}}(\log \log x)(\log \log \log x)$$
for all $x \geq C(b)$.
If $a = \log \log C(b)$, some of the possible values for $a$ and $b$ are given in the following table:
\medskip\par
\centering{
\begin{tabular}{|c|c|c|c|c|c|c|c|c|c|c|c|}
\hline
$a$ & 18.59&18.08&17.60&17.14&16.70&16.29&15.89&15.51&15.15&14.80&14.47\\
\hline
$b$&60&62&64&66&68&70&72&74&76&78&80\\
\hline
\end{tabular}
}
\end{thm}

As a consequence of Theorem \ref{Erdos1}, we get an explicit version of a theorem of Erd\"{o}s \cite{E} on the normal order of  $\omega(p-1)$.

\begin{cor}
\label{Erdos2}
Let $\eta>0$. Then for $x \geq \exp(\exp(a))$ we have 
$$
\frac{1}{\pi(x)} \#\{p\leq x;~\omega(p-1) >(1+\eta) \log\log{x} \} \leq \frac{1}{\eta^2}\left(\frac{251\log\log\log{x}}{60 \log\log{x}}+\frac{b}{\log\log{x}} \right),
$$
where some of the possible values for $a$ and $b$ are given in the table in Theorem \ref{Erdos1}. 
\end{cor}

The proofs of the above theorems follow the general strategy outlined in \cite[Section 3.3]{CM}. However, modifications are required since Corollary \ref{CEBVT} (our effective variant of the Bombieri-Vinogradov theorem) produces non-trivial results only for moduli bigger than a power of $\log{x}$.
A key idea is the application of the Brun-Titchmarsh inequality (Lemma \ref{BT}(f)) when Corollary \ref{CEBVT} is not useful. 
Generating the table in Theorem \ref{Erdos1} requires some care in choosing values for the parameters $U$ and $V$ that arise in the proof of the theorem. See Section \ref{Section3} for details.

Our second application is an effective version of a theorem of Goldfeld \cite[Theorem 1]{G}. Let $p_1$ and $p_2$ denote primes.
\begin{thm}
\label{Goldfeld 1}
There is an effective positive constant $C_0(\epsilon)$ depending only on $\epsilon>0$ such that 
$$
\frac{x}{2}-\left(26+\epsilon\right)\frac{x\log\log{x}}{\log{x}} < \sum_{p_1\leq x} \sum_{\substack{{x^{\frac{1}{2}}<p_2\leq x}\\{p_2\mid p_1-1}}} \log{p_2} < \frac{x}{2} + \left(13+{\epsilon}\right)\frac{x\log\log{x}}{\log{x}}
$$
for $x\geq C_0(\epsilon)$.
\end{thm}

By an argument similar to the proof of the lower bound in Theorem \ref{Goldfeld 1}, we can establish the following.  

\begin{thm}
\label{Proposition 1.8}
Suppose $\epsilon>0$ and $\frac{1}{2} \leq \theta <1-\frac{1}{2} \exp {\left(-\frac{1}{4} \right)}=0.6105\ldots$. Let 
$\delta(\theta)= \frac{1}{2} +2 \log \left( 2-2\theta \right)$. Then there is an effectively computable constant $C_0(\theta, \epsilon)$ such that
\begin{align} 
\sum_{p_1\leq x} \sum_{\substack{{x^{\theta}<p_2\leq x}\\{p_2\mid p_1-1}}} \log{p_2}
>  \delta(\theta) x - (26+\epsilon)\frac{x \log \log x}{\log{x}}
\end{align}
for $x \geq C_0(\theta, \epsilon)$.
\end{thm}

For the proofs of the previous two theorems, a key idea is that we apply the Brun-Titchmarsh inequality (Lemma \ref{BT}(f)) for the small moduli where our effective variant of the Bombieri-Vinogradov theorem (Corollary \ref{CEBVT}) gives only trivial estimates. 
Harman used this idea in \cite{H}.

Theorem \ref{Proposition 1.8} has applications to problems related to the least prime in an arithmetic progression,
the greatest prime divisor of a shifted prime, and the order of an integer modulo a prime. Let $L(p,a)$ be the least prime congruent to $a$ modulo $p$, let $P(n)$ be the greatest prime divisor of the integer $n$, and let $e_b(p)$ be the multiplicative order of the positive integer $b$ modulo $p$. 
\begin{cor}

\label{Corollary 1.9}

Let $n\geq 1$ be an integer, $\frac{1}{2} \leq \theta <1-\frac{1}{2} \exp {\left(-\frac{1}{4} \right)}=0.6105\ldots$, and $\delta(\theta) = \frac{1}{2} + 2\log(2-2\theta)$. Let $C_0(\theta, \epsilon)$ be the constant given in Theorem \ref{Proposition 1.8}.

\noindent \textbf{(i)}  Let $C_1(\theta, \epsilon)$ be the smallest value of $x$ such that
$$\frac{\log \log x}{\log x}< \frac{\delta(\theta)}{26+\epsilon}.$$
Then for every $x \geq \max \{C_0(\theta, \epsilon), C_1(\theta, \epsilon)\}$ there exists a prime $p$ with $x^\theta < p \leq x$ such that
$$L(p, 1)< p^{1/\theta}.$$

\noindent \textbf{(ii)} Let $C_2(\theta, \epsilon, n)$ be the smallest value of $x$ for which 
\begin{align*}
\frac{\log \log x}{\log x} < \frac{\delta(\theta)}{n(26+\epsilon)}.
\end{align*}
Then
\begin{align*}
\mathop{\sum_{p_1 \leq x}}_{P(p_1-1) > x^{\theta}} 1 > \delta(\theta)\left( 1-\frac{1}{n} \right)\frac{x}{\log x}
\end{align*}
whenever $x \geq \max\{C_0(\theta, \epsilon), C_2(\theta, \epsilon, n)\}.$

\noindent \textbf{(iii)} Let $C_3(\theta, \epsilon)$ be the smallest value of $x$ such that \eq{f1f2} holds, and let $C_4(\theta, \epsilon, n)$ be the smallest value of $x$ for which 
\begin{align*}
\frac{\log \log x}{\log x} < \frac{\delta(\theta)}{n(26+2/(1-\theta)+2\epsilon)}.
\end{align*}
Then 
\begin{align*}
\mathop{\sum_{p_1 \leq x}}_{e_2(p_1) > x^{\theta}} 1 > \delta(\theta)\left( 1-\frac{1}{n} \right)\frac{x}{\log x}
\end{align*}
whenever $x \geq \max\{C_0(\theta, \epsilon), C_3(\theta, \epsilon), C_4(\theta, \epsilon, n)\}$.
\end{cor}

The ineffective version of part (i) is due to Motohashi \cite{MO}.
Part (ii) is essentially Harman's effective estimate on the density of shifted primes with large prime divisors, which we discussed in the introduction.
The ineffective version of part (iii) is due to Goldfeld \cite{G}.

In the next corollary we give numerical values for all the constants in Corollary \ref{Corollary 1.9} when $\theta = 0.6$ and $n=100$. 
For part (i) we chose $\epsilon = 4$, and for parts (ii) and (iii) we chose $\epsilon = 2$.

\begin{cor}
\label{Corollary 1.10}
\noindent \textbf{(i)} For $x\geq\exp(\exp(8.47))$ there exists a prime $p$ with $x^{3/5} < p \leq x$ such that
$$L(p, 1) < p^{5/3}.$$

\noindent \textbf{(ii)} For $x\geq\exp(\exp(13.47))$
\begin{align*}
\mathop{\sum_{p \leq x}}_{P(p-1) > x^{0.6}} 1 > 0.0531 \frac{x}{\log x}.
\end{align*}

\noindent \textbf{(iii)} For $x\geq\exp(\exp(13.71))$
\begin{align*}
\mathop{\sum_{p \leq x}}_{e_2(p)>x^{0.6} } 1 > 0.0531 \frac{x}{\log x}.
\end{align*}
\end{cor}

Corollary \ref{Corollary 1.10} gives numerical values for the constants in Corollary \ref{Corollary 1.9} for specific values of the parameters. 
The constants $C_1(\theta,\epsilon)$, $C_2(\theta,\epsilon,n)$, $C_3(\theta, \epsilon)$, and $C_4(\theta,\epsilon,n)$ are clearly easy to calculate for fixed values of the parameters. Calculating $C_0(\theta,\epsilon)$ can be done by carefully following the proofs of Theorem \ref{Goldfeld 1} and Theorem \ref{Proposition 1.8} and incorporating the modification described in Remark \ref{remark 5.2}. For the calculation to be practical, some care must be taken in choosing values for the parameters $U$ and $V$ arising in these proofs. For part (i), where we chose $\theta = 0.6$, $n=100$, and $\epsilon = 4$, we set $U=(\log{x})^7$ and $V=x^{1/2}/(\log{x})^7$. For parts (ii) and (iii), where we chose $\theta = 0.6$, $n=100$, and $\epsilon = 2$, we set $U=(\log{x})^{6.75}$ and $V=x^{1/2}/(\log{x})^{6.75}$.

\begin{remarks}
(i) Results similar those in this section can be obtained for $\omega(p+a)$, $p_1 + a$,
$L(p,a)$, $P(p+a)$, and $e_b(p)$, for any fixed $a, b\in \mathbb{Z} \setminus \{0\}$. For simplicity in our presentation we described the results for $a=-1$ and $b=2$.

(ii) An adaptation of Theorem \ref{Erdos} to the case of $\omega(p+2)$ will give an explicit estimate for the number of twin primes $p\leq x$, however sieve methods may provide superior explicit results.

(iii) An effective version of the full Bombieri-Vinogradov theorem would imply the inequality of Theorem \ref{Erdos} with the right-hand side $(9+\epsilon)\log\log{x}$. The extra factor 
$\log\log\log{x}$ in our current theorem is a consequence of the fact that our variant is non-trivial only for $Q_1$ bigger than a power of $\log{x}$. Since we deal with moduli $\leq Q_1$ 
with the Brun-Titchmarsh inequality, the extra factor $\log\log\log{x}$ results. This illustrates a limitation of the method.

(iv) The constants appearing in the inequality of Theorem \ref{Goldfeld 1} are closely related to the exponent $\alpha$ of $\log{x}$ in Corollary \ref{CEBVT}. More precisely our proof establishes $2\alpha+\epsilon$ (resp. $4\alpha+\epsilon$) for the coefficient of $x\log\log{x}/\log{x}$ in the upper bound (resp. lower bound) given in Theorem \ref{Goldfeld 1}. So an improvement of the exponent $\alpha$ will improve the constants in Theorem \ref{Goldfeld 1}. An improvement of the exponent $\alpha$ will in fact improve the constants in all the theorems and corollaries in this section.
\end{remarks}


\section{Explicit Inequalities}

In the next lemma we will collect several known explicit inequalities (mostly from \cite{RS1} and \cite{Du0}) used in the proofs. Note that $\gamma=0.57721\ldots$ is Euler's constant and $\vartheta(x)=\sum_{p\leq x} \log{p}$.
\begin{lemma}
\label{BT}
\textbf{(a)} For $x \geq 10372$
\begin{equation}
\label{3.2.02}
-\frac{0.1}{\log x} - \frac{4}{15 \log^2 x}\leq\sum_{p\leq x} \frac{1}{p-1} - \log \log x - M \leq \frac{0.1}{\log x} + \frac{4}{15 \log^2 x}+\sum_{p} \frac{1}{p(p-1)},
\end{equation}
where 
$M = \gamma + \sum_{p} \rbr{\log\rbr{\frac{p-1}{p}} + \frac{1}{p}} = 0.26149\ldots$, and $\sum_{p} \frac{1}{p(p-1)}=0.57721\ldots$. The lower inequality holds for $x>1$.
 
\textbf{(b)} For $x \geq 2974$  
\begin{equation}
\label{log p}
-\frac{0.2}{\log x} - \frac{0.2}{ \log^2 x}\leq \sum_{p\leq x} \frac{\log{p}}{p-1} - \log  x - E \leq \frac{0.2}{\log x} + \frac{0.2}{ \log^2 x}+\sum_{p} \frac{\log{p}}{p(p-1)},
\end{equation}
where 
$E = -\gamma - \sum_{p} \frac{\log p}{p(p-1)} = -1.33258\ldots$. The lower inequality holds for $x>1$. Moreover, for $x\geq 8$

\begin{align}
\label{log p / (p-1)}
\sum_{p \leq x} \frac{\log p}{p-1} \leq \log x.
\end{align}
 
\textbf{(c)} For $x\geq 2973$
\begin{align}
\label{product}
\prod_{p\leq x} \left( 1-\frac{1}{p} \right) >\frac{e^{-\gamma}}{\log{x}} \left( 1-\frac{0.2}{\log^2{x}}\right).
 \end{align}

\textbf{(d)} For  $x\geq 563$
\begin{align}
\label{theta bound}
x - \frac{x}{2 \log x} < \vartheta(x) < x + \frac{x}{2 \log x}.
\end{align}
Moreover, for $x \geq 3594 641$
\begin{align}
\label{301}
x - 0.2 \frac{x}{\log^2 x} < \vartheta(x) < x + 0.2 \frac{x}{\log^2 x}.
\end{align}

\textbf{(e)} For $x\geq 17$
\begin{equation}
\label{pi lower bound}
\pi(x)>\frac{x}{\log{x}}.
\end{equation}
For $x > 1$
\begin{equation}\label{pi upper bound}
\pi(x) < 1.25506 \frac{x}{\log{x}}.
\end{equation}
For $x\geq 32299$
\begin{align} 
\label{pi lower}
\pi(x) \geq \frac{x}{\log{x}}+\frac{x}{\log^2{x}}+1.8 \frac{x}{\log^3{x}}.
\end{align}
For $x \geq 355991$
\begin{equation}
\label{3.2.01}
\pi(x) \leq \frac{x}{\log x} + \frac{x}{\log^2 x} + 2.51\frac{x}{\log^3 x}.
\end{equation}
 
\textbf{(f)} For all $1 \leq q < x$ and all integers $a$,
\begin{align}
\label{Brun-T}
\pi(x;q,a) < \frac{2x}{\phi(q) \log(x/q) }.
\end{align}

\textbf{(g)} For $x,V\geq 1$ we have
$$
\sum_{k\leq x} \left(\sum_{  \substack{{d\mid k} \\ {d\leq V}}  } \mu(d) \right)^2\leq \frac{4}{3} x (\log{e^3V})^2.
$$
\end{lemma}
\begin{proof}
{(a)} This is a consequence of the inequality
\begin{equation}
\label{3.2.022}
\left|\sum_{p\leq x} \frac{1}{p} - \log \log x - M \right|\leq \frac{0.1}{\log x} + \frac{4}{15 \log^2 x} \quad (x \geq 10372),
\end{equation}
(see \cite[Theorem 2]{Du0}). The lower inequality holds for $x>1$.

{(b)} We can deduce this from the inequality 
\begin{align}
\label{log p / p} 
\left|\sum_{p \leq x} \frac{\log p}{p}-\log x - E \right|\leq \frac{0.2}{\log x} + \frac{0.2}{ \log^2 x } \qquad (x \geq 2974), 
\end{align}
(see \cite[Theorem 3]{Du0}). The lower inequality holds for $x>1$.
By numerical calculation and \eqref{log p},
one readily establishes \eqref{log p / (p-1)}.

{(c)} This is \cite[Theorem 4]{Du0}.

{(d)} \eqref{theta bound} is \cite[Formulas (3.14) and (3.15)]{RS1}. \eqref{301} is given in 
\cite[p. 54]{Du0}.

{(e)} \eqref{pi lower bound} and \eqref{pi upper bound} are \cite[Corollary 1 of Theorem 2]{RS1}. \eqref{pi lower} 
and \eqref{3.2.01} are given in \cite[p. 55]{Du0}.

{(f)} This is the well-known Brun-Titchmarsh inequality, see \cite{MV} for a proof. 

{(g)} This follows immediately from \cite[Proposition 10.1]{GR}.
\end{proof}

\begin{remark}
We note that sharper estimates for the Chebyshev function $\psi(x)$ will improve many of the above estimates and consequently our results stated in Section \ref{sectiontwo}.
\end{remark}

\section{Proofs of Theorem \ref{Erdos} and Theorem \ref{Erdos1}}\label{Section3}



\begin{proof}
[Proof of Theorem \ref{Erdos}]
We start by observing that
$$\sum_{p\leq x} \omega(p-1)=\sum_{p\leq x} \sum_{\ell \mid p-1} 1 = \sum_{\ell \leq {x}}\pi(x; \ell, 1)$$
and
\begin{align} \label{3.2.4*}
\sum_{p \leq x} \omega^2(p-1)
=
\sum_{p \leq x} \mathop{\sum_{\ell_1 \divides p-1} \sum_{\ell_2 \divides p-1} }_{\ell_1 \neq \ell_2} 1 + \sum_{p \leq x} \sum_{\ell \divides p-1} 1
=
\mathop{\sum_{\ell_1 \leq x} \sum_{\ell_2 \leq x}}_{\ell_1 \neq \ell_2} \pi(x;\ell_1\ell_2,1) + \sum_{\ell \leq x} \pi(x;\ell,1).
\end{align} 
Let $0 < \epsilon \leq 16/3$, $T = (16 + \epsilon)/\epsilon$, $B \geq 5.5$, $U = (\log x)^B$, and $V = x^{1/T}/U$. 
We write
\begin{eqnarray}
\label{3.6*}
\notag
\mathop{\sum_{\ell_1 \leq x} \sum_{\ell_2 \leq x}}_{\ell_1 \neq \ell_2} \pi(x;\ell_1\ell_2,1) 
&=&
2 \mathop{\sum_{\ell_1 \leq V} \sum_{\ell_2 \leq U}}_{\ell_1 > \ell_2} \pi(x;\ell_1\ell_2,1)
+
\mathop{\sum \sum}_{\substack{U < \ell_1, \ell_2 \leq V \\ \ell_1 \neq \ell_2}} \pi(x;\ell_1\ell_2,1) \\
\notag
&&+
2 \mathop{ \sum_{\ell_1 \leq x} \sum_{V < \ell_2 \leq x} }_{\ell_1 < \ell_2} \pi(x;\ell_1\ell_2,1) \\
&=&
(I) + (II) + (III).
\end{eqnarray}

By applying \eqref{Brun-T} to $(I)$ and then employing \eqref{3.2.02} in the resulting expression we see that

\begin{eqnarray}
\label{3.2.6*}
\notag
(I) &<& \mathop{\sum_{\ell_1 \leq V} \sum_{\ell_2 \leq U}}_{\ell_1 > \ell_2} \frac{4x}{(\ell_1 - 1)(\ell_2 - 1)\log(x / \ell_1 \ell_2)} \\ 
\notag
&\leq& \frac{4x}{\log(x/UV)} \rbr{ \sum_{\ell_1 \leq V} \frac{1}{\ell_1 - 1} } \rbr{ \sum_{\ell_2 \leq U} \frac{1}{\ell_2 - 1} } \\
\notag
&=&  \frac{4x}{\log x} \frac{1}{1-T^{-1}} \rbr{ \sum_{\ell_1 \leq V} \frac{1}{\ell_1 - 1} } \rbr{ \sum_{\ell_2 \leq U} \frac{1}{\ell_2 - 1} } \\
&\leq& \frac{4x}{\log x} \frac{1}{1-T^{-1}}\left( \log \log{V} + g(V) \right)\left(\log \log{U} + g(U)\right)
\end{eqnarray}
whenever $U \geq 10372$ and $V \geq 10372$, where 
$$
g(x)=M+\frac{0.1}{\log{x}}+\frac{4}{15\log^2{x}}+\sum_p \frac{1}{p(p-1)}
$$
and $M$ is defined in Lemma \ref{BT}(a). 
We observe that the dominant term in \eqref{3.2.6*} is
$$ 
\frac{4x}{\log x} \frac{1}{1-T^{-1}}(\log\log{x})(\log\log\log{x}).
$$
Next we note that 
\begin{eqnarray}
\label{3.2.7*}
\notag
(II) &\leq& 
\mathop{\sum \sum}_{\substack{U < \ell_1, \ell_2 \leq V \\ \ell_1 \neq \ell_2}} \frac{\pi(x)}{\phi(\ell_1 \ell_2)} 
+ 2\sum_{\substack{q \leq V^2 \\ \ell(q) > U}} \abs{\pi(x;q,1) - \frac{\pi(x)}{\phi(q)}} \\
&\leq& \pi(x) \rbr{\sum_{U < \ell \leq V} \frac{1}{\ell-1}}^2 +  2\sum_{\substack{q \leq V^2 \\ \ell(q) > U}} \abs{\pi(x;q,1) - \frac{\pi(x)}{\phi(q)}}.
\end{eqnarray}
To estimate $(III)$ observe that if $p \leq x$ and $V^{T + 1} \geq x$, then $p-1$ can have at most $\lfloor T \rfloor$ distinct prime factors larger than $V$.
Thus
\begin{equation}
\label{3.2.8*}
(III) = 2\sum_{p \leq x} \mathop{\sum_{\substack{\ell_1 \leq x \\ \ell_1 \divides p-1}} \sum_{\substack{V < \ell_2  \leq x \\ \ell_2 \divides p-1}}}_{\ell_1 < \ell_2} 1
\leq 2T \sum_{p \leq x} \sum_{\substack{\ell_1 \leq x \\ \ell_1 \divides p-1}} 1 
= 2T \sum_{\ell \leq x} \pi(x;\ell,1) 
\end{equation}
when $x$ is large enough that $V^{T + 1} \geq x$.

Therefore, by \eq{3.2.4*}, \eq{3.6*}, \Comment{\eq{3.2.6},} \eq{3.2.7*}, and \eq{3.2.8*}, we have
\begin{align*}\begin{split}
\sum_{p \leq x} \rbr{\omega(p-1) - \log \log x}^2 
&= \sum_{p \leq x} \omega^2(p-1) - 2\log\log x \sum_{\ell \leq x} \pi(x;\ell,1) + \pi(x)(\log \log x)^2 \\
&\leq 
(I)
+ \pi(x) \rbr{\sum_{U < \ell \leq V} \frac{1}{\ell-1}}^2 
+ 2\sum_{\substack{q \leq V^2 \\ \ell(q) > U}} \abs{\pi(x;q,1) - \frac{\pi(x)}{\phi(q)}} \\
&\quad  + \left(1+2T-2\log\log x\right)\sum_{\ell \leq x} \pi(x;\ell,1) 
+ \pi(x) (\log \log x)^2. 
\end{split}
\end{align*}
Next we write $\sum_{\ell \leq x} \pi(x;\ell, 1)$ in the right-hand side of the above inequality as three sums over $\ell\leq U$, $U<\ell\leq V$, and $\ell>V$ respectively. We apply \eqref{Brun-T} in the first sum
and treat the third sum as we treated $(III)$.  Then after some rearrangement we arrive at
\begin{eqnarray}
\label{big-equation*}
\notag
\sum_{p \leq x} \rbr{\omega(p-1) - \log \log x}^2 
&\leq& 
(I) + \pi(x)\rbr{\sum_{U < \ell \leq V} \frac{1}{\ell-1} - \log \log x}^2  \\
\notag
&& + 2\sum_{\substack{q \leq V^2 \\ \ell(q) > U}} \abs{\pi(x;q,1) - \frac{\pi(x)}{\phi(q)}}
+ \left( 1+2T \right)\pi(x)\sum_{U < \ell \leq V} \frac{1}{\ell-1}
\\
\notag
&& + \left|1+2T-2\log\log x\right| \sum_{\substack{q \leq V \\ \ell(q) > U}} \abs{\pi(x;q,1) - \frac{\pi(x)}{\phi(q)}} \\
\notag
&& + \max\cbr{1+2T-2\log\log x,0} \frac{2x}{\log(x/U)}\sum_{\ell \leq U} \frac{1}{\ell - 1} \\
&& + \max\cbr{1+2T-2\log\log x,0} T\pi(x)
\end{eqnarray}
for all $x$ satisfying $U \geq 10372$, $V \geq 10372$, and $V^{T + 1} \geq x$.

It thus follows from \eq{3.2.6*}, our effective variant \eq{EBVTpi}, the estimate \eq{3.2.02} for $\sum_{\ell\leq x} \frac{1}{\ell-1}$, and the lower bound \eq{pi lower} for $\pi(x)$ that
\begin{align*}
\sum_{p\leq x} (\omega(p-1)-\log\log{x})^2 < (4+\epsilon)\pi(x)(\log\log{x})(\log \log \log x)
\end{align*}
for $x \geq C_0$, where $C_0$ can be explicitly determined for any given value of $\epsilon$ and $B$.
\end{proof}%

The proof of Theorem \ref{Erdos1} is very similar to the proof of Theorem \ref{Erdos}. 
The arguments only diverge in what is deduced from \eq{big-equation*}.
For Theorem \ref{Erdos1}, using \eq{EBVTpi}, \eq{3.2.02}, \eq{pi lower}, and \eq{3.2.6*}, we deduce from \eq{big-equation*} that
\begin{align}\label{little-equation*}
\sum_{p \leq x} \rbr{\omega(p-1) - \log \log x}^2 
< \rbr{4+\frac{\epsilon}{4} + \frac{b}{\log \log \log x}}\pi(x) (\log \log x)(\log \log \log x)
\end{align}
when $x$ is large enough in terms of $\epsilon$, $B$, and the additional parameter $b$. Setting $\epsilon = 11/15$ gives the result.

To generate the table in Theorem \ref{Erdos1}, we set $B = 7$ and, for each value of $b$, computed to a few decimal places the smallest $a$ for which the right-hand side of \eq{big-equation*} is majorized by the right-hand side of \eq{little-equation*} for all $x \geq \exp(\exp(a))$.
\begin{remark}
By taking $T = \sqrt{2\log \log \log x}$ in the proof of Theorem \ref{Erdos}, we can prove that
there is an effective constant $b_0$ such that for each $b > b_0$ there is an effective constant $C(b)$ such that 
$$\frac{1}{\pi(x)}\sum_{p\leq x} (\omega(p-1)-\log\log{x})^2 < 4(\log \log x)(\log \log \log x + \sqrt{2\log \log \log x} + b)$$
for all $x \geq C(b)$.
This statement is stronger than both Theorem \ref{Erdos} and \ref{Erdos1}.
However it is not practical to explicitly compute a value of $C(b)$ for small $b$.
\end{remark}

\section{Proofs of Theorem \ref{Goldfeld 1}, Theorem \ref{Proposition 1.8}, and Corollary \ref{Corollary 1.9}}\label{Section4}

We will need the following lemma in the proof of Theorem \ref{Goldfeld 1} and Theorem \ref{Proposition 1.8} below.

\begin{lemma}
\label{theta}
For $1/2\leq \theta\leq (1 + 1/\log(563))^{-1} = 0.86363\ldots$, we have
\begin{align}
\label{end}
\notag
\sum_{V < p \leq x^{\theta}} (\log p) \pi(x;p,1)
&<2x \log\left( \frac{\log(x/V)}{(1-\theta)\log{x}}\right) + \frac{2\theta x}{(1-\theta)V}
\\
& \quad 
+ \frac{x}{(\log x^\theta)\log (x/x^\theta)} + \frac{x}{(\log V)\log (x/V)} 
\end{align}
when $V \geq 563$. 
\end{lemma}
\begin{proof}
We first apply the Brun-Titchmarsh inequality \eqref{Brun-T} to the sum on the left of inequality \eqref{end} and then split the resulting sum into two to get
\begin{align}
\label{estimates}
\notag
\sum_{V < p \leq x^{\theta}} (\log p) \pi(x;p,1) 
&< \sum_{V < p \leq x^{\theta}} \frac{2x \log p}{(p-1) \log (x/p)} \\
&= \sum_{V < p \leq x^{\theta}} \frac{2x \log p}{p \log (x/p)} + \sum_{V < p \leq x^{\theta}} \frac{2x \log p}{p(p-1) \log (x/p)}.
\end{align}
By \eqref{log p / (p-1)}, for the second sum on the right we have
\begin{align}
\label{f1}
\sum_{V < p \leq x^{\theta}} \frac{2x \log p}{p(p-1) \log (x/p)}
< \frac{2x}{V\log(x/x^{\theta})} \sum_{V < p \leq x^{\theta}} \frac{\log p}{p-1} < \frac{2x \log x^{\theta}}{V\log(x/x^{\theta})} = \frac{2\theta x}{(1-\theta)V} 
\end{align}
for $x^\theta \geq 8$.  To treat the first sum on the right of \eqref{estimates}, we set $f(t) = (t \log (x/t))^{-1}$. Then by partial summation we have
\begin{align*}
\sum_{V < p \leq x^{\theta}} \frac{\log p}{p \log (x/p)} 
&= \vartheta(x^{\theta}) f(x^{\theta}) - \vartheta(V) f(V) - \int_{V}^{x^{\theta}} \vartheta(t) f^{\prime}(t) dt,
\end{align*}
where $\vartheta(x)=\sum_{p\leq x} \log{p}$.
It follows from the bounds for $\vartheta(x)$ in \eqref{theta bound} that, as long as $V \geq 563$,
\begin{align}
\label{f2}
\sum_{V < p \leq x^{\theta}} \frac{\log p}{p \log (x/p)} 
< \int_{V}^{x^{\theta}} f(t) dt +  \frac{x^{\theta} f(x^{\theta})}{2 \log x^{\theta}} + \frac{V f(V)}{2 \log V} 
+ \int_{V}^{x^{\theta}} \frac{t f^{\prime}(t)}{2 \log t} dt.
\end{align}
Since $f^{\prime}(t) = - (f(t))^2 (\log(x/t) - 1)$ and because $563\geq e^{\theta/(1-\theta)}$ for $\theta\leq (1 + 1/\log(563))^{-1}$, we have
\begin{align}
\label{f3}
\int_{V}^{x^{\theta}} \frac{t f^{\prime}(t)}{2 \log t} dt \leq 0
\end{align}
for $V \geq 563$.
Noting $\frac{d}{dt} \left(\log \log (x/t)\right) = -f(t)$, we have
\begin{align}
\label{f4}
\int_{V}^{x^{\theta}} f(t) dt 
=\log\left( \frac{\log(x/V)}{(1-\theta)\log{x}}\right).
\end{align}
Applying \eqref{f3} and \eqref{f4} in \eqref{f2} and then employing the resulting inequality together with \eqref{f1} in \eqref{estimates} gives the result.
\end{proof}

%
\begin{proof}[Proof of Theorem \ref{Goldfeld 1}]
We start by observing that
\begin{align}
\label{basic decomp}
\sum_{p_1\leq x} \sum_{\substack{{x^{\frac{1}{2}}<p_2\leq x}\\{p_2\mid p_1-1}}} \log{p_2}
&=
\sum_{p_1\leq x} \sum_{p_2\mid (p_1-1)} \log{p_2}
-
\sum_{p_1\leq x} \sum_{\substack{{p_2\leq x^{1/2}}\\{p_2\mid (p_1-1)}}} \log{p_2}.
\end{align}
Next we decompose the first sum on the right as  
\begin{eqnarray}\label{E formula 2}
\notag
\sum_{p_1 \leq x} \sum_{p_2 \mid (p_1 - 1)} \log p_2 
&= &\sum_{p_1 \leq x} \mathop{ \sum_{p_{2}^{k} \mid (p_1 - 1)} }_{k \geq 1} \log p_2 
- \sum_{p_1 \leq x} \mathop{ \sum_{p_{2}^{k} \mid (p_1 - 1)} }_{k \geq 2} \log p_2  \\
\notag
&= &\sum_{p_1 \leq x} \log(p_1 - 1) - \sum_{p_1 \leq x} \mathop{ \sum_{p_{2}^{k} \mid (p_1 - 1)} }_{k \geq 2} \log p_2 \\
\notag
&=& \sum_{p_1 \leq x}\log p_1 - \sum_{p_1 \leq x} \log \rbr{\frac{p_1}{p_1 - 1}} - \sum_{p_1 \leq x} \mathop{ \sum_{p_{2}^{k} \mid (p_1 - 1)} }_{k \geq 2} \log p_2 \\
&= &\vartheta(x) - E_1 - E_2.
\end{eqnarray}
We will first bound this quantity from below. 
By \eqref{301} we have
\begin{align}
\label{302}
E_1 
< \log \log x + \gamma - \log\rbr{1-\frac{0.2}{\log^2 x}} \qquad (x \geq 2973).
\end{align}
Observe that
\begin{align*}
E_2 = \sum_{k \geq 2} \sum_{p \leq x^{1 / k}} (\log p) \pi(x;p^k,1).
\end{align*}
Using Lemma \ref{BT}(f) and the trivial bound $\pi(x;q,1) \leq x/q$, we find
\begin{eqnarray}
\label{E_1}
\notag
E_2 
&\leq& \sum_{k \geq 2} \sum_{p \leq x^{1/2k}} \frac{2 x \log p}{\phi(p^k) \log(x/p^k)  } 
+ \sum_{k \geq 2} \sum_{x^{1/2k} < p \leq x^{1/k}} \frac{x \log p}{p^k} 
\\
\notag
&\leq& \frac{4x}{\log x} \sum_{k \geq 2} \sum_{p \leq x^{1/2k}} \frac{\log p}{\phi(p^k) } 
+  \sum_{x^{1/4} < p \leq x^{1/2}} \frac{x^{3/4} \log p}{p}
+  \sum_{3 \leq k \leq \frac{\log x}{\log 2}} \sum_{x^{1/2k} < p \leq x^{1/k}} \frac{x^{2/3} \log p}{p} \\
&<& \frac{4x}{\log x} \sum_{p} \frac{\log p}{(p-1)^2} 
+ x^{3/4} \sum_{p \leq x^{1/2}} \frac{\log p}{p}
+ \frac{ x^{2/3} (\log x) }{2} \sum_{p \leq x^{1/3}} \frac{1}{p}.
\end{eqnarray}
We next apply \eq{3.2.022}, \eq{log p / p}, 
and the fact that
$$
\sum_p \frac{\log p}{(p-1)^2} = \frac{5}{4} - 0.02303\ldots,
$$
in \eqref{E_1} to obtain an upper bound for $E_2$.  Applications of the resulting upper bound for $E_2$, the upper bound \eqref{302} for $E_1$, and the lower bound \eqref{301} for $\vartheta(x)$ in \eqref{E formula 2}
yield
\begin{align} \label{302.1}
\sum_{p_1 \leq x} \sum_{p_2 \mid (p_1 -1)} \log p_2
> x - \frac{5x}{\log x} \qquad (x \geq 3.9595567809\times 10^{15}).
\end{align}

Next we consider the second sum on the right of \eqref{basic decomp}, i.e.
$$\sum_{p_1\leq x} \sum_{\substack{p_2 \mid (p_1-1)\\ p_2\leq x^{1/2}}} \log p_2.$$ 
Let $U = (\log x)^{6.5 }$ and $V = x^{1/2} / (\log x)^{{6.5} }$. Write
\begin{align}
\label{U formula 2}
\sum_{p_1 \leq x} \mathop{ \sum_{p_2 \leq x^{1/2}} }_{p_2 \mid (p_1 - 1)} \log p_2 
=
\sum_{ p \leq x^{1/2}} (\log p) \pi(x;p,1)
= 
\sum_{p \leq U}  + \sum_{U < p \leq V} + \sum_{V < p \leq x^{1/2}}.
\end{align}
We need to find upper bounds on the three sums on the right.

By the Brun-Titchmarsh inequality \eqref{Brun-T},
we have 
\begin{align}
\label{start}
\sum_{p \leq U} (\log p) \pi(x,p,1) 
< \sum_{p \leq U} \frac{2x \log p}{(p-1) \log (x/p)}
\leq \frac{2x}{\log(x/U)} \sum_{p \leq U} \frac{\log p}{p-1}.
\end{align}
Observe that \eq{start} together with \eqref{log p} imply that
\begin{align*}
\sum_{p \leq U} (\log p) \pi(x,p,1) 
< 13\frac{x \log \log{x}}{\log{x}}
\end{align*}
provided $x \geq C_1$, where $C_1$ is an explicitly computable constant.
For the sum with $V < p \leq x^{1/2}$, we use Lemma \ref{theta} with $\theta = \frac{1}{2}$ to get
\begin{align*} 
\label{end theta=1/2}
\sum_{V<p \leq x^{1/2}} (\log p) \pi(x;p,1) 
&\leq
2x\log \rbr{1 + \frac{\log(x / V^2)}{\log x}  } + \frac{4x}{(\log x)^2}
\\
\notag
& \quad  + \frac{x}{(\log V)\log (x/V)} + \frac{2 x}{V}
\end{align*}
for $V \geq 563$. 
So by employing the inequality $\log(1+y)\leq y-\frac{y^2}{2}+\frac{y^3}{3},$
valid for $y>-1$, we can find an explicit constant $C_2$ such that
\begin{equation*}
\sum_{V<p \leq x^{1/2}} (\log p) \pi(x,p,1) 
< {26} \frac{x \log \log{x}}{\log{x}}
\end{equation*}
for $x\geq C_2$.
For the sum over $U < p \leq V$  in (\ref{U formula 2}), we will use our effective variant of the Bombieri-Vinogradov theorem.
Applying (\ref{EBVTpi}) yields
\begin{equation} \label{303}
\sum_{U < p \leq V} (\log p) \abs{ \pi(x;p,1) - \frac{\pi(x)}{\phi(p)} } 
< c_2 \rbr{ 8\frac{x}{U} + 18 \frac{x^{\frac{11}{12}}}{U^{\frac{1}{2}}} +  5x^{\frac{5}{6}} \log\rbr{\frac{ex^{1/2}}{U^2}} } (\log x)^{\frac{9}{2}} (\log V).
\end{equation}
By \eqref{log p} we have
\begin{align} \label{306}
\sum_{U<p\leq V} \frac{\log{p}}{\phi(p)}\leq \log{\left(V/U\right)}+\frac{0.2}{\log{V}}+\frac{0.2}{\log^2{V}}+\frac{0.2}{\log{U}}+\frac{0.2}{\log^2{U}}+\sum_{p}\frac{\log{p}}{p(p-1)}
\end{align}
for $V>U\geq 2974$. Let $\epsilon > 0$ be given.
Then by \eq{303}, \eqref{306}, the upper bound for $\pi(x)$ given in \eq{3.2.01}, and the triangle inequality we can obtain
\begin{align*}
\sum_{U < p \leq V} (\log p) \pi(x,p,1) < \frac{x}{2} - \rbr{13-\frac{\epsilon}{2}}{\frac{x \log \log x}{\log x}}
\end{align*}
for $x \geq C_3(\epsilon)$, where $C_3(\epsilon)$ is explicitly computable.

Now that we have the lower bound \eq{302.1} and the upper bounds on the three sums on the right of \eq{U formula 2}, we can put everything together to conclude 
\begin{align*}
\sum_{p_1\leq x} \sum_{\substack{{x^{\frac{1}{2}}<p_2\leq x}\\{p_2\mid p_1-1}}} \log{p_2}
> 
\frac{x}{2} - (26+\epsilon) \frac{x \log \log x}{\log x}
\end{align*}
for $x \geq C_4(\epsilon)$, where $C_4(\epsilon)$ can be determined explicitly.

The upper bound in the theorem is proved by a similar argument. 
Let $U = (\log x)^{6.5}$ and $V = x^{1/2} / (\log x)^{6.5}$. By \eq{E formula 2} and \eq{301} we have
\begin{align} \label{307}
\sum_{p_1\leq x} \sum_{p_2\mid (p_1-1)} \log{p_2} 
< \vartheta(x)
< x + 0.2 \frac{x}{\log^2 x} \qquad (x \geq 3594641).
\end{align}
By \eq{U formula 2} we have
\begin{eqnarray} 
\label{308}
\notag
\sum_{p_1\leq x} \sum_{\substack{{p_2\leq x^{1/2}}\\{p_2\mid (p_1-1)}}} \log{p_2}
&\geq &
\sum_{U < p \leq V} (\log p) \pi(x;p,1) \\
&\geq&
- \sum_{U < p \leq V} (\log p) \abs{\pi(x;p,1) - \frac{\pi(x)}{\phi(p)}} + \sum_{U < p \leq V} (\log p) \frac{\pi(x)}{\phi(p)}.
\end{eqnarray}
We estimate the first sum in the last line using (\ref{EBVTpi}). To estimate the second sum, we combine the lower bound \eqref{pi lower} on $\pi(x)$ with the observation that \eq{log p} implies
\begin{align} \label{310}
\sum_{U < p \leq V} \frac{\log p}{\phi(p)}
>
 \log(V/U) - \sum_{p}\frac{\log p}{p(p-1)} - \frac{0.2}{\log V} - \frac{0.2}{\log^2 V} - \frac{0.2}{\log U} - \frac{0.2}{\log^2 U},
\end{align}
for  $V>U \geq 2974$. Let $\epsilon > 0$ be given.
By combining \eq{303}, \eq{307}, \eq{308}, \eq{310}, \eqref{pi lower}, and (\ref{EBVTpi}) we find that
\begin{align*}  
\sum_{p_1\leq x} \sum_{\substack{{x^{\frac{1}{2}}<p_2\leq x}\\{p_2\mid p_1-1}}} \log{p_2}
< \frac{x}{2} + (13+{\epsilon}) \frac{x \log \log x}{\log x}
\end{align*}
for $x \geq C_5(\epsilon)$.
\end{proof}
%
%

%
%
%
\begin{proof}[Proof of Theorem \ref{Proposition 1.8}]
Write
\begin{eqnarray*}
\sum_{ x^\theta< p \leq x} (\log p) \pi(x;p,1)&=& \sum_{p_1\leq x} \sum_{p_2\mid p_1-1} \log{p_2}
- \sum_{ p \leq U} (\log p) \pi(x;p,1)\\
&& -\sum_{ U<p \leq V} (\log p) \pi(x;p,1)- \sum_{ V<p \leq x^\theta } (\log p) \pi(x;p,1),
\end{eqnarray*}
where $U = (\log x)^{6.5}$ and $V = x^{1/2} / (\log x)^{6.5}$. Then imitate the proof of the lower bound in Theorem \ref{Goldfeld 1}.
\end{proof}

\begin{remark}\label{remark 5.2}
The proofs of Theorem \ref{Goldfeld 1} and Theorem \ref{Proposition 1.8} can also be written by choosing $U=(\log{x})^{6.5+b}$ and $V=x^{1/2}/(\log{x})^{6.5+b}$, where $0\leq b<\epsilon/4$.
\end{remark}

From now until the end of this section we assume that $\epsilon>0$, $n$ is a positive integer, $\frac{1}{2} \leq \theta <1-\frac{1}{2} \exp {\left(-\frac{1}{4} \right)}=0.6105\ldots$, and $C_0(\theta, \epsilon)$ is the constant given in Theorem \ref{Proposition 1.8}.

\begin{proof}[Proof of Corollary \ref{Corollary 1.9}]
(i) Since $\theta < 1 - \frac{1}{2}\exp(-\frac{1}{4})$, we have $\delta(\theta) > 0$.  So, for a fixed choice of $\theta$, by Theorem \ref{Proposition 1.8} we have 
$\sum_{ x^\theta< p \leq x} (\log p) \pi(x;p,1) > 0$ for all $x \geq \max \{C_0(\theta, \epsilon), C_1(\theta, \epsilon)\}$.


%
\medskip\par

(ii) For $p_1 \leq x$, we have
\begin{align*}
\log x 
> \log (p_1 - 1) 
\geq \mathop{ \sum_{p_2 \mid (p_1 - 1)} }_{p_2 > x^{\theta}} \log p_2,
\end{align*}
hence
\begin{align*}
\mathop{\sum_{p_1 \leq x}}_{P(p_1-1) > x^{\theta}} \log x 
> 
\mathop{\sum_{p_1 \leq x}}_{P(p_1-1) > x^{\theta}} \mathop{ \sum_{p_2 \mid (p_1 - 1)} }_{p _2 > x^{\theta}} \log p_2
= 
\sum_{p_1 \leq x} \mathop{ \sum_{p_2 \mid (p_1 - 1)} }_{x^{\theta} <p_2\leq x} \log p_2
=
\sum_{ x^\theta< p \leq x} (\log p) \pi(x;p,1).
\end{align*}
The result follows by applying Theorem \ref{Proposition 1.8} and then dividing by $\log x$.
%

\medskip\par

(iii) Since $e_2(p) \divides p-1$, we have
\begin{align*} 
\sum_{p_1 \leq x} \sum_{\substack{ x^{\theta}<p_2\leq x \\ p_2 \divides e_2(p_1)  }} \log p_2
=
\sum_{p_1 \leq x} \sum_{\substack{ x^{\theta}<p_2\leq x \\ p_2 \divides p_1 - 1  }} \log p_2
-
\sum_{p_1 \leq x} \sum_{\substack{ x^{\theta}<p_2\leq x \\ p_2 \divides p_1 - 1 \\ p_2 \ndivides e_2(p_1)  }} \log p_2.
\end{align*}
By Theorem \ref{Proposition 1.8} there is an effective constant $C_0(\theta, \epsilon)$ such that 
\begin{align} \label{3.7.0}
\sum_{p_1 \leq x} \sum_{\substack{ x^{\theta}<p_2\leq x \\ p_2 \divides e_2(p_1)  }} \log p_2
\geq \delta(\theta) x -(26+\epsilon) \frac{x\log\log{x}}{\log{x}}-\sum_{p_1 \leq x} \sum_{\substack{ x^{\theta}<p_2\leq x \\ p_2 \divides p_1 - 1 \\ p_2 \ndivides e_2(p_1)  }} \log p_2
\end{align}
for all $x \geq C_0(\theta, \epsilon)$. Next we observe that
\begin{align} \label{3.7.1}
\sum_{p_1 \leq x} \sum_{\substack{ x^{\theta}<p_2\leq x \\ p_2 \divides p_1 - 1 \\ p_2 \ndivides e_2(p_1)  }} \log p_2
=
\sum_{p_1 \leq x} \sum_{\substack{ x^{\theta}<p_2\leq x^{\theta} \log x \\ p_2 \divides p_1 - 1 \\ p_2 \ndivides e_2(p_1)  }} \log p_2
+
\sum_{p_1 \leq x} \sum_{\substack{ x^{\theta}\log x <p_2\leq x \\ p_2 \divides p_1 - 1 \\ p_2 \ndivides e_2(p_1)  }} \log p_2.
\end{align}
In the first sum on the right of \eqref{3.7.1} we apply the Brun-Titchmarsh inequality \eqref{Brun-T} and \eq{log p} to deduce 
\begin{eqnarray}
\label{3.7.2}
\notag
\sum_{p_1 \leq x} \sum_{\substack{ x^{\theta}< p_2 \leq x^{\theta} \log x \\ p_2 \divides p_1 - 1 \\ p_2 \ndivides e_2(p_1)  }} \log p_2
&\leq&
\sum_{p_1 \leq x} \sum_{\substack{ x^{\theta}<p_2\leq x^{\theta} \log x \\ p_2 \divides p_1 - 1   }} \log p_2
\quad =
\sum_{ x^{\theta} < p \leq x^{\theta} \log x  } (\log p) \pi(x,p,1) \\
\notag
&<&
\sum_{ x^{\theta} < p \leq x^{\theta} \log x  } \frac{2x \log p}{(p-1)\log(x/p)}
\leq
\frac{2x}{\log(x^{1-\theta} / \log x)} \sum_{ x^{\theta} < p \leq x^{\theta} \log x  } \frac{\log p}{p-1} \\
\notag
&\leq&
\frac{2x}{\log(x^{1-\theta} / \log x)} \left( \log \log x + \sum_{p} \frac{\log p}{p(p-1)} 
+ \frac{0.2}{\log(x^{\theta} \log x)}  \right. \\
\notag
&& \hspace{4 cm} \left.
+ \frac{0.2}{\log^2(x^{\theta} \log x)}
+ \frac{0.2}{\theta \log x} + \frac{0.2}{\theta^2 \log^2 x} \right) \\
&=& \frac{2}{1-\theta} \frac{x \log \log x}{\log x} + f_1(\theta, x)
\end{eqnarray}
whenever $x^{\theta}\log x \geq 2974$.  

Now we deal with the second sum on the right of \eq{3.7.1}.  Let 
$$
M_2^{\theta}(x) =\sum_{p_1 \leq x} \sum_{\substack{ x^{\theta}\log x < p_2 \leq x \\ p_2 \divides p_1 - 1 \\ p_2 \ndivides e_2(p_1)  }} 1.
$$
For each pair of primes $p_1,p_2$ counted in $M_2^{\theta}(x)$, 
$$
2^{(p_1-1)/p_2} \equiv 1 \pmod{p_1}.
$$
So
$$
2^{M_2^{\theta}(x)} \leq \prod_{m \leq x^{1-\theta} / \log x} (2^m - 1)
$$
and therefore 
$$
M_2^{\theta}(x) \leq  \sum_{m \leq x^{1-\theta} / \log x} m \leq  \frac{x^{1-\theta}}{2\log x} \rbr{\frac{x^{1-\theta}}{\log x} + 1}.
$$
Thus
\begin{align} \label{3.7.3}
\sum_{p_1 \leq x} \sum_{\substack{ x^{\theta}\log x < p_2 \leq x \\ p_2 \divides p_1 - 1 \\ p_2 \ndivides e_2(p_1)  }} \log p_2 
\leq M_2^{\theta}(x) \log x
\leq \frac{1}{2 } \rbr{\frac{x^{2(1-\theta)}}{\log x} + x^{1-\theta}}=f_2(x, \theta).
\end{align}
Let $C_3(\theta, \epsilon)$ be the smallest $x$ such that
\begin{equation}
\label{f1f2}
x^\theta \log x \geq 2974   
\qquad \text{and} \qquad
f_1(\theta, x)+f_2(\theta, x) 
\leq 
\epsilon \frac{x\log\log{x}}{\log{x}},
\end{equation}
where $f_1(\theta, x)$ and $f_2(\theta, x)$ are defined in \eq{3.7.2} and \eq{3.7.3}.
Next by employing \eq{3.7.2} and \eq{3.7.3} together with \eq{f1f2} in \eq{3.7.1}  and applying the resulting inequality in \eq{3.7.0} we deduce that 
$$
\sum_{p_1 \leq x} \sum_{\substack{ x^{\theta}<p_2\leq x \\ p_2 \divides e_2(p_1)  }} \log p_2 
> 
\delta(\theta) x - (26+\frac{2}{1-\theta}+2\epsilon)\frac{x \log \log x}{\log{x}} 
$$
for $x \geq \max \{ C_0(\theta,\epsilon), C_3(\epsilon) \}$. Upon observing that
$$
\mathop{\sum_{p_1 \leq x}}_{e_2(p_1) > x^{\theta}} \log x 
> 
\sum_{p_1 \leq x} \mathop{ \sum_{\substack{x^{\theta} <p_2\leq x\\ p_2 \mid e_2(p_1 )}} } \log p_2,
$$
the result follows immediately.
\end{proof}

\section{Proof of Theorem \ref{TMVT}}
Fix arbitrary real numbers $Q > 0$ and $x \geq 4$.
In this section, we shall establish \eq{MVT}, which is the key inequality used in the proof of Theorem \ref{EBVT}.  
We closely follow Vaughan's proof as described in \cite[Chapter 28]{D}, and in each step will make the constants explicit.  

The main tool in the proof of \eq{MVT} is the large sieve inequality
\begin{align} \label{LSI}
\sum_{q \leq Q} \frac{q}{\phi(q)} \mathop{{\sum}^{\ast}}_{\chi \hspace{0.15 em} (q)} \abs{ \sum_{m = m_0 + 1}^{m_0 + M} a_m \chi(m) }^2 
\leq (M + Q^2) \sum_{m = m_0 + 1}^{m_0 + M} \abs{a_m}^2.
\end{align}
Here the $a_m$'s are arbitrary complex numbers.
Note that (\ref{LSI}) is given in \cite[p. 561, last line]{Mo}.

We will use the following consequence of \eqref{LSI}. 
\begin{lemma}\label{LSLeq inequality}
Suppose that $a_{m_0},\cdots, a_M$ and $b_{n_0}, \cdots, b_N$ are complex numbers. Then
\begin{gather} \label{LSLeq}
\notag
\sum_{q \leq Q} \frac{q}{\phi(q)} \mathop{{\sum}^{\ast}}_{\chi \hspace{0.15 em} (q)} \max_{y} \abs{ \mathop{ \sum_{m = m_0}^{M} \sum_{n = n_0}^{N} }_{mn \leq y} a_m b_n \chi(mn) }  \\
\leq 
c_3 (M^{\prime} + Q^2)^{\frac{1}{2}} (N^{\prime} + Q^2)^{\frac{1}{2}} \rbr{\sum_{m = m_0}^{M} \abs{a_m}^2 }^{\frac{1}{2}} \rbr{\sum_{n = n_0}^{N} \abs{b_n}^2 }^{\frac{1}{2}} \log(2MN),
\end{gather}
where 
$$
c_3 = \frac{2}{\pi} \left( \frac{2 + \log( \log(2) /\log(4/3))}{\log 2} \right) = 2.64456\ldots,
$$
$M^{\prime} = M - m_0 + 1$ and $N^{\prime} = N - n_0 + 1$ are the number of terms in the sums over $m$ and $n$, respectively, and the maximum is taken over all real numbers $y$.
\end{lemma}

\begin{proof}
Since the product $mn$ will be a positive integer not exceeding $MN$, we may assume that $y$ is of the form $y = k + \frac{1}{2}$, where $k \in \cbr{1,\ldots,MN}$.  By contour integration one can show that if $\alpha \geq 0$, $\beta > 0$, and $T > 0$, then 
$$
\int_{-T}^{T} e^{- i t \alpha} \frac{\sin(t\beta)}{t} dt 
= 
\left\{ 
     \begin{array}{ll}
          \pi + O^*(2T^{-1}|\alpha - \beta|^{-1}) & \text{if $\alpha < \beta$}, \\
          O^*(2T^{-1}|\alpha - \beta|^{-1}) & \text{if $\alpha > \beta$},
     \end{array}
  \right.
$$
where $O^*$ means that the implied constants are less than or equal to 1.
By setting $\alpha = \log (mn)$, $\beta = \log y$, and summing against $a_m b_n \chi(mn)$, we find
\begin{align} \label{LSLPeq1}
\notag
\pi \mathop{ \sum_{m = m_0}^{M} \sum_{n = n_0}^{N} }_{mn \leq y} a_m b_n \chi(mn) 
&= \sum_{m = m_0}^{M} \sum_{n = n_0}^{N} a_m b_n \chi(mn) \int_{-T}^{T} (mn)^{-it} \frac{\sin(t \log y)}{t} dt \\
&\qquad + \sum_{m = m_0}^{M} \sum_{n = n_0}^{N} a_m b_n \chi(mn) O^{*}(2T^{-1}|\log(mn / y)|^{-1}).
\end{align}
Because $y$ is a half integer $\leq MN + \frac{1}{2}$, we have
$$
\abs{\log\rbr{\frac{mn}{y}}} \geq \frac{\log(4/3)}{MN} \qquad \text{ and } \qquad \abs{\sin(t \log y)} \leq \min(1,\abs{t}\log(2MN)).
$$
Therefore the modulus of the left-hand side of (\ref{LSLPeq1}) is 
$$
\leq \int_{-T}^{T} \abs{ \sum_{m=m_0}^{M} \sum_{n=n_0}^{N} \frac{a_m}{m^{it}} \frac{b_n}{n^{it}} \chi(mn) }
\min \rbr{ \frac{1}{\abs{t}}, \log( 2MN ) } dt \\
+
\frac{2}{ \log \left( \frac{4}{3} \right) } \frac{MN}{T} \sum_{m=m_0}^{M} \sum_{n=n_0}^{N} \abs{a_m b_n},
$$
and so the left-hand side of (\ref{LSLeq}) is
\begin{align} \label{LSLPeq2}
&\leq \frac{1}{\pi} \int_{-T}^{T} { \sum_{q \leq Q} \frac{q}{\phi(q)} \mathop{{\sum}^{\ast}}_{\chi \hspace{0.15 em} (q)} \abs{ \sum_{m=m_0}^{M} \sum_{n=n_0}^{N} \frac{a_m}{m^{it}} \frac{b_n}{n^{it}} \chi(mn) } }
\min \rbr{ \frac{1}{\abs{t}}, \log( 2MN ) } dt \notag\\
&\qquad +
\frac{2}{ \pi \log \left( \frac{4}{3} \right) } \frac{M N Q^2}{T} \sum_{m=m_0}^{M} \sum_{n=n_0}^{N} \abs{a_m b_n}.
\end{align}
Applying first Cauchy's inequality and then (\ref{LSI}), the sum under the integral in \eqref{LSLPeq2} is seen to be
\begin{align*}
&\leq
\rbr{ \sum_{q \leq Q} \frac{q}{\phi(q)} \mathop{{\sum}^{\ast}}_{\chi \hspace{0.15 em} (q)} \abs{ \sum_{m=m_0}^{M}  \frac{a_m}{m^{it}} \chi(m) }^2   }^{\frac{1}{2}}
\rbr{ \sum_{q \leq Q} \frac{q}{\phi(q)} \mathop{{\sum}^{\ast}}_{\chi \hspace{0.15 em} (q)} \abs{ \sum_{n=n_0}^{N}  \frac{b_n}{n^{it}} \chi(n) }^2   }^{\frac{1}{2}} \\
&\leq
(M^{\prime}+ Q^2)^{\frac{1}{2}} (N^{\prime}+ Q^2)^{\frac{1}{2}} 
\rbr{\sum_{m = m_0}^{M} \abs{a_m}^2 }^{\frac{1}{2}} \rbr{\sum_{n = n_0}^{N} \abs{b_n}^2 }^{\frac{1}{2}}.
\end{align*}
Thus, by applying Cauchy's inequality in the second sum in (\ref{LSLPeq2}), the right-hand side of (\ref{LSLPeq2}) is
\begin{align*}
&\leq \frac{1}{\pi} (M^{\prime}+ Q^2)^{\frac{1}{2}} (N^{\prime}+ Q^2)^{\frac{1}{2}} 
\rbr{\sum_{m = m_0}^{M} \abs{a_m}^2 }^{\frac{1}{2}} \rbr{\sum_{n = n_0}^{N} \abs{b_n}^2 }^{\frac{1}{2}} \int_{-T}^{T} \min \rbr{ \frac{1}{\abs{t}}, \log( 2MN ) } dt \\
&\qquad + 
\frac{2}{ \pi \log \left( \frac{4}{3} \right) } \frac{M^{\frac{3}{2}} N^{\frac{3}{2}} Q^2}{T} \rbr{ \sum_{m=m_0}^{M} \abs{a_m}^2 }^{\frac{1}{2}} \rbr{ \sum_{n=n_0}^{N}  \abs{b_n}^2 }^{\frac{1}{2}}.
\end{align*}
The integral in the above inequality is $2\log(eT\log(2MN))$.  So, since $Q^2 \leq (M^{\prime}+ Q^2)^{\frac{1}{2}} (N^{\prime}+ Q^2)^{\frac{1}{2}}$, to complete the proof it will be enough to show that
\begin{align} \label{c3.1}
\frac{2}{\pi}
\rbr{ \log(eT\log(2MN)) + \frac{(MN)^{3/2}}{\log(4/3)T} }
\frac{1}{\log(2MN)}
\leq c_3
\end{align}
for some choice of $T$.
A bit of calculus reveals that the left-hand side of \eq{c3.1} is minimized by choosing $T = (MN)^{3/2}/\log(4/3)$.  With this choice of $T$ the left-hand side is decreasing in $MN$, and so (since $MN \geq 1$) its maximum value is seen to be $c_3$.
\end{proof}
We are almost at the heart of the proof of (\ref{MVT}), but first we reduce to the case $2 \leq Q \leq x^{1/2}$.  If $Q < 1$ the sum on 
the left-hand side of (\ref{MVT}) is empty and there is nothing to prove.  If $1 \leq Q < 2$ then only the $q = 1$ term appears and we have
\begin{equation}
\label{Chebyshev}
\sum_{q \leq Q} \frac{q}{\phi(q)} \mathop{{\sum}^{\ast}}_{\chi \hspace{0.15 em} (q)} \max_{y \leq x} \abs{ \psi(y,\chi) } = \max_{y \leq x} \abs{ \sum_{n \leq y} \Lambda(n)  } = \psi(x) \leq A_0 x
\end{equation}
which is better than (\ref{MVT}).  Note the last inequality is the Chebyshev estimate (see \cite[Theorem 12]{RS1}) where $A_0= \max (\psi(x)/x)= \psi(113)/113$.  
In the case $Q > x^{1/2}$, (\ref{MVT}) follows from (\ref{LSLeq}) with $M = m_0 = n_0 = 1$, $N = \lfloor x \rfloor$, $a_m = 1$, $b_n = \Lambda(n)$, and the estimate
$$
\sum_{n \leq x} \Lambda(n)^2 \leq \psi(x) \log x \leq A_0 x \log x.
$$

From now on we assume $2 \leq Q \leq x^{1/2}$.  We will need the following decomposition of $\Lambda(n)$ due to Vaughan (see \cite{V4} or \cite[p. 139]{D}). Let $U$ and $V$ be arbitrary real numbers $\geq 1$. We have
$$
\Lambda(n) = \lambda_1(n) + \lambda_2(n) + \lambda_3(n) + \lambda_4(n),
$$
where 
\begin{eqnarray*}
\lambda_{1}(n) &=&
\left\{
\begin{array}{ll}
\Lambda(n) & \mbox{if \, $n \leq U$}, \\
0 & \mbox{if \, $n > U$},
\end{array}
\right.  \\
\lambda_{2}(n) &=& \mathop{ \sum_{hd = n} }_{d \leq V} \mu(d) \log h, \\
\lambda_{3}(n) &=& - \mathop{\mathop{\sum_{mdr = n}}_{m \leq U}}_{d \leq V} \Lambda(m) \mu(d), 
\end{eqnarray*}
and
\begin{eqnarray*}
\lambda_{4}(n) &=& - \mathop{ \mathop{ \sum_{mk = n} }_{m > U} }_{k > V} \Lambda(m)  \mathop{ \sum_{d \mid k} }_{d \leq V} \mu(d).
\end{eqnarray*}
Assume $y \leq x$, $q \leq Q$, and $\chi$ is a character mod $q$.  We use the above decomposition to write
$$
\psi(y,\chi) = S_1 + S_2 + S_3 + S_4,
$$
where
$$
S_i = \sum_{n \leq y} \lambda_i(n) \chi(n).
$$
We take $U$ and $V$ to be functions of $x$ and $Q$ to be specified later.  
We will treat each $S_i$ separately.  It immediately follows that
\begin{align*} 
\abs{S_1} \leq \sum_{n \leq U} \Lambda(n) \leq A_0U
\end{align*}
by the Chebyshev estimate \eqref{Chebyshev}.  Then
\begin{align} \label{S1}
\sum_{q \leq Q} \frac{q}{\phi(q)} \mathop{{\sum}^{\ast}}_{\chi \hspace{0.15 em} (q)} \max_{y \leq x} \abs{ S_1 }
\leq A_0U \sum_{q \leq Q} \frac{q}{\phi(q)} \mathop{{\sum}^{\ast}}_{\chi \hspace{0.15 em} (q)} 1 
\leq 
A_0UQ^2.
\end{align}

Next we write
$$
S_2 
= \sum_{n \leq y} \chi(n) \mathop{ \sum_{hd = n} }_{d \leq V} \mu(d) \log h 
= \mathop{ \sum_{hd \leq y} }_{d \leq V} \mu(d) \chi(hd) \log h
= \sum_{d \leq V} \mu(d) \chi(d) \sum_{h \leq y/d} \chi(h) \log h.
$$
We have
\begin{align*}
S_2 &= \sum_{d \leq V} \mu(d) \chi(d) \sum_{h \leq y/d} \chi (h) \int_{1}^{h} \frac{dw}{w} 
= \sum_{d \leq V} \mu(d) \chi(d) \int_{1}^{y/d} \sum_{w \leq h \leq y/d} \chi(h) \frac{dw}{w} \\
&= \int_{1}^{y} \sum_{d \leq \min \rbr{ V,y/w } } \mu(d) \chi(d) \sum_{w \leq h \leq y/d} \chi(h) \frac{dw}{w} = \int_{1}^{y} \rbr{  \sum_{d \leq V } \mu(d) \chi(d) \sum_{w \leq h \leq y/d} \chi(h) } \frac{dw}{w}.
\end{align*}
In the last expression we dropped the condition $d \leq y/w$ because the sum over $h$ is empty when $d > y/w$.  Now we see that
\begin{align*}
\abs{S_2} 
\leq \int_{1}^{y} \rbr{\sum_{d \leq V} \abs{ \sum_{w \leq h \leq y/d} \chi(h) } } \frac{dw}{w} 
\leq  \rbr{\sum_{d \leq V} \max_{1\leq w \leq y} \abs{ \sum_{w \leq h \leq y/d} \chi(h) } } \int_{1}^{y} \frac{dw}{w}
\end{align*}
and so
\begin{align} \label{S2step1}
\abs{S_2} \leq (\log y) \sum_{d \leq V} \max_{1\leq w \leq y} \abs{ \sum_{w \leq h \leq y/d} \chi(h) }.
\end{align}

If $q = 1$ then the only character $\chi$ mod $q$ is the trivial character, and so
$$
\abs{S_2} \leq (\log y) \sum_{d \leq V} \max_{1 \leq w \leq y} \abs{ \sum_{w \leq h \leq y/d} 1 } = y (\log y) \sum_{d \leq V} \frac{1}{d} <  y (\log y) (\log eV) < x (\log xV)^2.
$$

If $q > 1$ and $\chi$ is a primitive character mod $q$, the Polya-Vinogradov inequality
$$
\abs{\sum_{a \leq n \leq b} \chi(n) } < q^{\frac{1}{2}} \log q
$$
holds for all $a$ and $b$  \cite[pp. 135-136]{D}; hence
\begin{equation}
\label{PV}
\abs{S_2} < (\log y) \sum_{d \leq V} \max_{1 \leq w \leq y} \rbr{q^{\frac{1}{2}}\log q} 
\leq q^{\frac{1}{2}} V (\log y) (\log q) <  q^{\frac{1}{2}} V (\log xV)^{2}.
\end{equation}
These two estimates for $\abs{S_2}$ imply
\begin{align} \label{S2}
\notag
\sum_{q \leq Q} \frac{q}{\phi(q)} \mathop{{\sum}^{\ast}}_{\chi \hspace{0.15 em} (q)} \max_{y \leq x} \abs{ S_2 }
&= 
\mathop{{\sum}^{\ast}}_{\chi \hspace{0.15 em} (1)} \max_{y \leq x} \abs{ S_2 } 
+ \sum_{1 < q \leq Q} \frac{q}{\phi(q)} \mathop{{\sum}^{\ast}}_{\chi \hspace{0.15 em} (q)} \max_{y \leq x} \abs{ S_2 } \\
\notag
&<
x (\log xV)^2 + V (\log xV)^2 \sum_{1 < q \leq Q} \frac{q^{\frac{3}{2}}}{\phi(q)} \mathop{{\sum}^{\ast}}_{\chi \hspace{0.15 em} (q)} 1 \\
&<
( x + Q^{\frac{5}{2}}V ) (\log xV)^2.
\end{align}

Next we consider $S_3$.  We write
\begin{align*}
S_3 
&
= - \mathop{\mathop{\sum_{mdr \leq y}}_{m \leq U}}_{d \leq V} \Lambda(m) \mu(d) \chi(mdr) 
= - \sum_{tr \leq y}  \mathop{ \mathop{ \sum_{md = t} }_{m \leq U} }_{d \leq V} \Lambda(m) \mu(d) \chi(tr).
\end{align*}
Consider the last expression.  Since the innermost sum is empty when $t > UV$, we are free to introduce the condition $t \leq UV$ on the sum over $t$.  
Hence
$$
S_3 = - \sum_{t \leq UV} \sum_{r \leq y/t}  \mathop{ \mathop{ \sum_{md = t} }_{m \leq U} }_{d \leq V} \Lambda(m) \mu(d) \chi(t) \chi(r).
$$
Now we split the summation over small and large $t$ by writing
$$
- S_3 = S_3^{\prime} + S_3^{\prime \prime} = \sum_{t \leq U} + \sum_{U < t \leq UV}.
$$
Since
\begin{align} \label{at}
 |{\mathop{ \mathop{ \sum_{md = t} }_{m \leq U} }_{d \leq V} \Lambda(m) \mu(d) \chi(t)}  |
\leq \sum_{m \divides t} \Lambda(m)
= \log t,
\end{align}
we have
\begin{align*} 
\abs{S_{3}^{\prime}} \leq (\log U) \sum_{t \leq U} \left |{\sum_{r \leq y/t} \chi(r)} \right |.
\end{align*}
Note the similarity to the inequality (\ref{S2step1}) which we obtained for $S_2$.
Indeed, by proceeding 
as we did with $S_2$, we obtain
\begin{align} \label{S3prime}
\sum_{q \leq Q} \frac{q}{\phi(q)} \mathop{{\sum}^{\ast}}_{\chi \hspace{0.15 em} (q)} \max_{y \leq x} \abs{ S_{3}^{\prime} }
<
( x + Q^{\frac{5}{2}}U )  (\log xU)^2.
\end{align}

We deal now with $S_{3}^{\prime \prime}$ by breaking the sum for $U < t \leq UV$ into sums over dyadic sub-intervals. We have
$$
S_{3}^{\prime \prime}
=  \mathop{\sum_{M = 2^{\alpha}}}_{\frac{1}{2}U < M \leq UV} \mathop{\sum_{U < t \leq UV}}_{M < t \leq 2M} 
\sum_{r \leq y/t}  \mathop{ \mathop{ \sum_{md = t} }_{m \leq U} }_{d \leq V} \Lambda(m) \mu(d)  \chi(tr).
$$
Since $M < t$ and $y \leq x$, we have $r \leq x/M$ in the sum over $r$.  Then 
\begin{align*}
\sum_{q \leq Q} \frac{q}{\phi(q)} \mathop{{\sum}^{\ast}}_{\chi \hspace{0.15 em} (q)} \max_{y \leq x} \abs{S_{3}^{\prime \prime}}
&\leq
\mathop{\sum_{M = 2^{\alpha}}}_{\frac{1}{2}U < M \leq UV} \sum_{q \leq Q} \frac{q}{\phi(q)} \mathop{{\sum}^{\ast}}_{\chi \hspace{0.15 em} (q)} \max_{y \leq x} \abs{ \sum_{\substack{U < t \leq UV \\ M < t \leq 2M}}
\sum_{\substack{ r \leq x/M \\ rt \leq y}}   a_t b_r \chi(tr) },
\end{align*}
where we have put
$$
a_t = \mathop{ \mathop{ \sum_{md = t} }_{m \leq U} }_{d \leq V} \Lambda(m) \mu(d) \qquad \text{and} \qquad b_r = 1.
$$
%
Then (\ref{LSLeq}) implies 
\begin{align*}
\sum_{q \leq Q} \frac{q}{\phi(q)} \mathop{{\sum}^{\ast}}_{\chi \hspace{0.15 em} (q)} \max_{y \leq x} \abs{S_{3}^{\prime \prime}}
&\leq
\mathop{\sum_{M = 2^{\alpha}}}_{\frac{1}{2}U < M \leq UV} 
c_3 (T^{\prime} + Q^2)^{\frac{1}{2}} (R^{\prime} + Q^2)^{\frac{1}{2}} 
\rbr{ \mathop{  \sum_{U < t \leq UV} }_{M < t \leq 2M} \abs{a_t}^{2} }^{\frac{1}{2}} \\
& \hspace{3.5 cm} \times \rbr{ \sum_{r \leq x/M} \abs{b_r}^{2} }^{\frac{1}{2}} \log \rbr{  2 \min(UV,2M)\frac{x}{M} },
\end{align*} 
where $T^{\prime}$ and $R^{\prime}$ are the number of terms in the sums over $t$ and $r$, respectively.
Clearly, $T^{\prime} \leq M$, $R^{\prime} \leq x / M$, and $\log(2\min(UV,2M)x/M) \leq \log(4x)$. By (\ref{at}), $\abs{a_t} \leq \log t$. 
Putting these facts together we see that
\begin{align*}
\sum_{q \leq Q} \frac{q}{\phi(q)} \mathop{{\sum}^{\ast}}_{\chi \hspace{0.15 em} (q)} \max_{y \leq x} \abs{S_{3}^{\prime \prime}}
&\leq
 \mathop{\sum_{M = 2^{\alpha}}}_{\frac{1}{2}U < M \leq UV}
c_3 \rbr{M + Q^2}^{\frac{1}{2}}\rbr{\frac{x}{M} + Q^2}^{\frac{1}{2}} 
M^{\frac{1}{2}} (\log 2M) \rbr{\frac{x}{M}}^{\frac{1}{2}} (\log 4x) \\
&\leq
\mathop{\sum_{M = 2^{\alpha}}}_{\frac{1}{2}U < M \leq UV}
c_3 \rbr{x + Q x^{\frac{1}{2}} M^{\frac{1}{2}} + QxM^{-\frac{1}{2}} + Q^2 x^{\frac{1}{2}} }(\log 2M)(\log 4x).
\end{align*}
Using that
\begin{align*} 
\mathop{\sum_{M = 2^{\alpha}}}_{\frac{1}{2}U < M \leq UV} 1
\leq
\frac{\log (2UV)}{\log 2},
\end{align*}
we conclude
\begin{align} \label{S3primeprime}
\sum_{q \leq Q} \frac{q}{\phi(q)} \mathop{{\sum}^{\ast}}_{\chi \hspace{0.15 em} (q)} \max_{y \leq x} \abs{S_{3}^{\prime \prime}}
&< \frac{c_3}{\log 2} \rbr{x + Q x^{\frac{1}{2}} U^{\frac{1}{2}} V^{\frac{1}{2}} +  2^{\frac{1}{2}} Q x U^{-\frac{1}{2}}  + Q^2 x^{\frac{1}{2}} } 
(\log 2UV)^2 (\log 4x).
\end{align}

Next we treat $S_4$.  We have
\begin{align*}
S_4 
=
- \sum_{\substack{mk \leq y \\ m > U \\ k > V}} \Lambda(m)  \left( \sum_{\substack{d \mid k \\ d \leq V}} \mu(d) \right)  \chi(mk) 
=
- \sum_{U < m \leq x/V} \sum_{\substack{V < k \leq x/M \\ mk \leq y}} \Lambda(m) \left( \sum_{\substack{d \mid k \\ d \leq V}} \mu(d) \right) \chi(mk).
\end{align*}
We employ the same technique we used with $S_{3}^{\prime \prime}$. Writing $S_4$ as a dyadic sum we have
$$
S_4 = - \sum_{\substack{M = 2^{\alpha} \\ \frac{1}{2}U < M \leq x/V}} \sum_{\substack{U < m \leq x/V \\ M < m \leq 2M}} \sum_{\substack{V < k \leq x/M \\ mk \leq y}} \Lambda(m) \left( \sum_{\substack{d \mid k \\ d \leq V}} \mu(d) \right) \chi(mk).
$$
By the triangle inequality
\begin{align*}
\sum_{q \leq Q} \frac{q}{\phi(q)} \mathop{{\sum}^{\ast}}_{\chi \hspace{0.15 em} (q)} \max_{y \leq x} \abs{ S_4 } 
&\leq
\sum_{\substack{M = 2^{\alpha} \\ \frac{1}{2}U < M \leq x/V}}
\sum_{q \leq Q} \frac{q}{\phi(q)} \mathop{{\sum}^{\ast}}_{\chi \hspace{0.15 em} (q)} \max_{y \leq x}
 \abs{ \sum_{\substack{U < m \leq x/V \\ M < m \leq 2M}}  
 \sum_{\substack{V < k \leq x/M \\ mk \leq y}} a_m b_k \chi(mk) },
\end{align*}
where 
\begin{align*}
a_m = \Lambda(m) \qquad \text{and} \qquad b_k = \mathop{ \sum_{d \mid k} }_{d \leq V} \mu(d).
\end{align*}
Then applying (\ref{LSLeq}) gives
\begin{align*}
\sum_{q \leq Q} \frac{q}{\phi(q)} \mathop{{\sum}^{\ast}}_{\chi \hspace{0.15 em} (q)} \max_{y \leq x} \abs{S_{4}}
&\leq
\mathop{\sum_{M = 2^{\alpha}}}_{\frac{1}{2}U < M \leq x/V } 
c_3 (M^{\prime} + Q^2)^{\frac{1}{2}} (K^{\prime} + Q^2)^{\frac{1}{2}} 
\rbr{ \sum_{\substack{U < m \leq x/V \\ M < m \leq 2M}} \abs{a_m}^{2} }^{\frac{1}{2}}  \\
& \hspace{3 cm} \times \rbr{ \sum_{\substack{V < k \leq x/M \\ mk\leq y}} \abs{b_k}^{2} }^{\frac{1}{2}} \log \rbr{  2 \min \rbr{\frac{x}{V},2M} \frac{x}{M} },
\end{align*}
where $M^{\prime}$ and $K^{\prime}$ denote the number of terms in the sums over $m$ and $k$, respectively.  Evidently, $M^{\prime} \leq M$, $K^{\prime} \leq x/M$, and $\log( 2 \min \rbr{x/V,2M} x/M ) \leq \log(4x)$. Chebyshev's estimate \eqref{Chebyshev} yields
$$
\mathop{\sum_{U < m \leq x/V}}_{M < m \leq 2M} |a_m|^2 \leq \sum_{m \leq 2M} \Lambda(m)^2 \leq \psi(2M) \log 2M \leq 2A_0M\log 2M.
$$
By  Lemma \ref{BT}(g),
\begin{equation}
\label{norm2}
\sum_{V < k \leq x/M} |b_k|^2 \leq \frac{4}{3}\frac{x}{M} \rbr{\log e^3 V}^2.
\end{equation}
By combining these estimates we find that $\sum_{q \leq Q} \frac{q}{\phi(q)} \mathop{{\sum}^{\ast}}_{\chi \hspace{0.15 em} (q)} \max_{y \leq x} \abs{S_4}$ is 
\begin{align*}
&\leq
\mathop{\sum_{M = 2^{\alpha}}}_{\frac{1}{2}U < M \leq x/V}
c_3 \rbr{M + Q^2}^{\frac{1}{2}}\rbr{\frac{x}{M} + Q^2}^{\frac{1}{2}} 
 (2A_0)^{\frac{1}{2}} M^{\frac{1}{2}} (\log 2M)^{\frac{1}{2}} 
 \rbr{ \frac{x}{M} }^{\frac{1}{2}} \rbr{\log e^3 V} (\log 4x) \\
&<
\frac{2^{\frac{3}{2}}A_0^{\frac{1}{2}}c_3 }{3^{\frac{1}{2}}}\rbr{x + Q x V^{-\frac{1}{2}} + 2^{\frac{1}{2}} Q x U^{-\frac{1}{2}}  + Q^2 x^{\frac{1}{2}} } 
\rbr{\log \rbr{\frac{2x}{V}}}^{\frac{1}{2}} \rbr{\log {e^3 V}} (\log 4x) 
\mathop{\sum_{M = 2^{\alpha}}}_{\frac{1}{2}U < M \leq x/V} 1.
\end{align*}
Then since
\begin{align*} 
\mathop{\sum_{M = 2^{\alpha}}}_{\frac{1}{2}U < M \leq x/V} 1 
\leq
\frac{\log(2x/V)}{\log 2},
\end{align*}
we have
\begin{eqnarray}
\label{S4}
\notag
\sum_{q \leq Q} \frac{q}{\phi(q)} \mathop{{\sum}^{\ast}}_{\chi \hspace{0.15 em} (q)} \max_{y \leq x} \abs{S_4} 
&<& 
\frac{2^{\frac{3}{2}}A_0^{\frac{1}{2}}c_3 }{3^{\frac{1}{2}} \log 2}
\rbr{x + Q x V^{-\frac{1}{2}} + 2^{\frac{1}{2}} Q x U^{-\frac{1}{2}} + Q^2 x^{\frac{1}{2}} } \\
&&\qquad \times \rbr{\log \rbr{\frac{2x}{V}}}^{\frac{3}{2}} (\log e^3 V)  (\log 4x).
\end{eqnarray}
Since
$$
\psi(y,\chi) = S_1 + S_2 + S_{3}^{\prime} + S_{3}^{\prime \prime} + S_4,
$$
combining the estimates (\ref{S1}), (\ref{S2}), (\ref{S3prime}), (\ref{S3primeprime}), and (\ref{S4}) gives 
\begin{align*}
\sum_{q \leq Q} \frac{q}{\phi(q)} \mathop{{\sum}^{\ast}}_{\chi \hspace{0.15 em} (q)} \max_{y \leq x} \abs{ \psi(y,\chi) } 
< 
c_4 K L,
\end{align*}
where 
\begin{gather*}
c_4 = \max \rbr{ A_0, \frac{c_3}{\log 2}, \frac{2^{\frac{3}{2}}A_0^{\frac{1}{2}}c_3 }{3^{\frac{1}{2}} \log 2} } = \frac{2^{\frac{3}{2}}A_0^{\frac{1}{2}}c_3 }{3^{\frac{1}{2}} \log 2}, \\
K = 4x + 2Q^2 x^{\frac{1}{2}} + Q^2U + Q^{\frac{5}{2}}U + Q^{\frac{5}{2}}V  + 2^{\frac{3}{2}} Q x U^{-\frac{1}{2}} + Q x V^{-\frac{1}{2}} + Q x^{\frac{1}{2}} U^{\frac{1}{2}} V^{\frac{1}{2}}, \\
L = \max\cbr{
(\log xV)^2,
(\log xU)^2,
(\log 2UV)^2 (\log 4x),
\rbr{\log \rbr{\frac{2x}{V}}}^{\frac{3}{2}} \rbr{\log e^3 V} (\log 4x)
}.
\end{gather*}
At this point, we specify $U$ and $V$. If $x^{\frac{1}{3}} \leq Q \leq x^{\frac{1}{2}}$, we set $U = V = Q^{-1}x^{\frac{2}{3}}$; hence
\begin{align*}
K &= 4x + 2Q^2x^{\frac{1}{2}} + Qx^{\frac{2}{3}} + Q^{\frac{3}{2}}x^{\frac{2}{3}} + Q^{\frac{3}{2}}x^{\frac{2}{3}} + 2^{\frac{3}{2}} Q^{\frac{3}{2}}x^{\frac{2}{3}} + Q^{\frac{3}{2}}x^{\frac{2}{3}} + x^{\frac{7}{6}} \\
&\leq 4x + 2Q^2x^{\frac{1}{2}} + (3+2^{\frac{3}{2}})Q^{\frac{3}{2}}x^{\frac{2}{3}} + 2Qx^{\frac{5}{6}}
\end{align*}
and
\begin{align*}
L \leq 2\rbr{\frac{4}{3}}^{\frac{3}{2}} \rbr{\frac{1}{3} + \frac{3}{2\log2}} (\log x)^{\frac{7}{2}}
\end{align*}
If $Q \leq x^{\frac{1}{3}}$, we set $U = V = x^{\frac{1}{3}}$, which leads to
\begin{align*}
K &= 4x + 2Q^2x^{\frac{1}{2}} + Q^2x^{\frac{1}{3}} + Q^{\frac{5}{2}}x^{\frac{1}{3}} + Q^{\frac{5}{2}}x^{\frac{1}{3}} + 2^{\frac{3}{2}}Qx^{\frac{5}{6}} + Qx^{\frac{5}{6}} + Qx^{\frac{5}{6}} \\
&\leq 4x + 2Q^2x^{\frac{1}{2}} + 3Q^{\frac{3}{2}}x^{\frac{2}{3}} + (2+2^{\frac{3}{2}})Qx^{\frac{5}{6}}
\end{align*}
and 
\begin{align*}
L \leq 2\rbr{\frac{7}{6}}^{\frac{3}{2}} \rbr{\frac{1}{3} + \frac{3}{2\log2}} (\log x)^{\frac{7}{2}}.
\end{align*}
So for $Q$ in either range we have
\begin{gather*}
\sum_{q \leq Q} \frac{q}{\phi(q)} \mathop{{\sum}^{\ast}}_{\chi \hspace{0.15 em} (q)} \max_{y \leq x} \abs{ \psi(y,\chi) } \\
< 2\rbr{\frac{4}{3}}^{\frac{3}{2}} \rbr{\frac{1}{3} + \frac{3}{2\log2}} c_4
\rbr{4x + 2Q^2 x^{\frac{1}{2}} + 6Q^{\frac{3}{2}}x^{\frac{2}{3}} + 5Qx^{\frac{5}{6}}} 
(\log{x})^{\frac{7}{2}},
\end{gather*}
which is precisely (\ref{MVT}).

\section{Proof of Theorem \ref{EBVT}}

\begin{proof}
Let $y \geq 2$ and $(a,q)=1$.  By orthogonality of characters modulo $q$, we have
$$
\psi(y;q,a) = \frac{1}{\phi(q)} \sum_{\chi} \overline{\chi}(a) \psi(y,\chi).
$$
Put
$$
\psi^{\prime}(y,\chi) 
= 
\left\{
\begin{array}{ll}
\psi(y,\chi) & \text{if } \chi \not = \chi_0, \\
\psi(y,\chi) - \psi(y) & \text{if } \chi = \chi_0.
\end{array}
\right.
$$
Here $\chi_0$ denotes the principal character mod $q$. Then
\begin{align*}
\psi(y,q,a) - \frac{\psi(y)}{\phi(q)} = \frac{1}{\phi(q)} \sum_{\chi} \overline{\chi}(a) \psi^{\prime}(y,\chi).
\end{align*}
For a character $\chi$ (mod $q$), we let $\chi^{\ast}$ be the primitive character modulo $q^*$ inducing $\chi$.
We have
$$
\psi^{\prime}(y,\chi^{\ast}) - \psi^{\prime}(y,\chi) 
= \psi(y,\chi^{\ast}) - \psi(y,\chi)
= \sum_{p^k \leq y} (\log p) (\chi^{\ast}(p^k) - \chi(p^k)).
$$
If $p \ndivides q$ then $(p^k,q^{\ast}) = 1$, and hence $\chi^{\ast}(p^k) = \chi(p^k)$.  If $p \divides q$ then $\chi(p^k) = 0$.  Therefore 
$$
\abs{ \psi^{\prime}(y, \chi^{\ast}) - \psi^{\prime}(y,\chi) } 
\leq \mathop{ \sum_{p^k \leq y} }_{p \divides q} (\log p) \leq (\log y) \sum_{p \divides q} 1 \leq (\log qy)^2.
$$
It follows that
$$
\abs{\psi(y,q,a) - \frac{\psi(y)}{\phi(q)}} 
\leq \frac{1}{\phi(q)} \sum_{\chi} \abs{\psi^{\prime}(y,\chi)}
\leq (\log qy)^2 + \frac{1}{\phi(q)} \sum_{\chi} \abs{\psi^{\prime}(y,\chi^{\ast})}.
$$
From here we see that the left-hand side of \eq{EBVTpsi} is 
$$
\leq Q(\log Qx)^2 + \sum_{\substack{q \leq Q \\ \ell(q)>Q_1}}\frac{1}{\phi(q)} \sum_{\chi} \max_{2 \leq y \leq x} \abs{\psi^{\prime}(y,\chi^{\ast})}.
$$
The first term is evidently smaller than the expression on the right of \eq{EBVTpsi} without the constant $c_1$.  So we just need to show that the second term is smaller than 
$$
(c_1-1) \rbr{ 4\frac{x}{Q_1} + 4x^{\frac{1}{2}}Q + 18x^{\frac{2}{3}}Q^{\frac{1}{2}} + 5x^{\frac{5}{6}} \log \rbr{ \frac{eQ}{Q_1} } } \rbr{\log x}^{5}.
$$
Since a primitive character $\chi^{\ast}$ (mod $q^{*}$) induces characters to moduli which are multiples of~$q^{\ast}$ and since $\psi^{\prime}(y,\chi^{\ast}) = 0$ whenever $\chi$ is principal, we have
\begin{align*}
\sum_{\substack{q \leq Q \\ \ell(q)>Q_1}}\frac{1}{\phi(q)} \sum_{\hspace{-0 em}\chi \hspace{0.15 em} (q)} \max_{2 \leq y \leq x} \abs{\psi^{\prime}(y,\chi^{\ast})}
&= \sum_{\substack{q \leq Q \\ \ell(q)>Q_1 }} \frac{1}{\phi(q)} \sum_{\substack{q^{\ast} \divides q \\ q^{\ast}\neq 1}} \mathop{{\sum}^{\ast}}_{\hspace{-0 em}\chi \hspace{0.15 em} (q^{\ast})} \max_{2 \leq y \leq x} \abs{\psi^{\prime}(y,\chi)} \\
&= \sum_{\substack{kq^{\ast}\leq Q \\ {\ell(kq^{\ast})>Q_1, q^{\ast} \neq 1}} } \frac{1}{\phi(kq^{\ast})}  \mathop{{\sum}^{\ast}}_{\hspace{-0 em}\chi \hspace{0.15 em} (q^{\ast})} \max_{2 \leq y \leq x} \abs{\psi^{\prime}(y,\chi)} \\
&\leq \sum_{\substack{q^{\ast} \leq Q\\ \ell(q^{\ast})>Q_1}} \mathop{{\sum}^{\ast}}_{\hspace{-0 em}\chi \hspace{0.15 em} (q^{\ast})} \max_{2 \leq y \leq x} \abs{\psi^{\prime}(y,\chi)} \Comment{\rbr }  \sum_{k \leq \frac{Q}{ q^{*}} \Comment{\substack{ k \leq \frac{Q}{ q^{*}}\\\ell(k)>Q_1}}  } \frac{1}{\phi(kq^{\ast})} .
\end{align*}
We shall now observe that
$$
\sum_{k \leq x}\frac{1}{\phi(k)} \leq E_0 \log (ex)
$$
for $x>0$ (the proof is essentially Exercises 4.4.11, 4.4.12, and 4.4.13 in \cite{M}).  As $q^{\ast} \leq Q \leq x^{1/2}$ and $\phi(k)\phi(q^{\ast}) \leq \phi(kq^{\ast})$, we have
$$
\sum_{ k \leq \frac{Q}{ q^{*}} \Comment{\substack{ k \leq \frac{Q}{ q^{*}}\\\ell(k)>Q_1}} } \frac{1}{\phi(kq^{\ast})} \leq \sum_{k \leq x^{\frac{1}{2}}} \frac{1}{\phi(kq^{\ast})}
\leq \frac{1}{\phi( q^{\ast} )} \sum_{ k \leq x^{\frac{1}{2}} } \frac{1}{\phi(k)} 
\leq \frac{E_0}{\phi( q^{\ast} )} \log (ex^{\frac{1}{2}})
< \frac{5 E_0}{4 \phi( q^{\ast} )} \log x,
$$
where the last inequality assumes $x \geq 4$.  Hence
$$
\sum_{\substack{q \leq Q \\ \ell(q)>Q_1}}\frac{1}{\phi(q)} \sum_{\hspace{-0 em}\chi \hspace{0.15 em} (q)} \max_{2 \leq y \leq x} \abs{\psi^{\prime}(y,\chi^{\ast})} 
\leq \frac{5 E_0}{4} (\log x) \sum_{\substack{ q \leq Q \\ \ell(q)>Q_1}}\frac{1}{\phi(q)} \mathop{{\sum_{\hspace{-0 em}\chi \hspace{0.15 em} (q)}}^{\ast}} \max_{2 \leq y \leq x} \abs{\psi^{\prime}(y,\chi)}.
$$
For $q>1$ if $\chi$ is a primitive character (mod $q$) then $\chi$ is non-principal and $\psi(y,\chi) = \psi^{\prime}(y,\chi)$.  Therefore, since $Q_1 \geq 1$, we can replace $\psi^{\prime}(y,\chi)$ by $\psi(y,\chi)$ on the right-hand side of the last inequality.  
Thus it suffices to show
$$
\sum_{\substack{q \leq Q \\\ell(q)>Q_1}}\frac{1}{\phi(q)} \mathop{{\sum_{\chi}}^{\ast}} \max_{2 \leq y \leq x} \abs{\psi(y,\chi)}
\leq \frac{4(c_1-1)}{5 E_0}  \rbr{ 4\frac{x}{Q_1} + 4x^{\frac{1}{2}}Q + 18x^{\frac{2}{3}}Q^{\frac{1}{2}} + 5x^{\frac{5}{6}} \log \rbr{ \frac{eQ}{Q_1} } } \rbr{\log x}^{\frac{7}{2}}.
$$
We will actually prove something slightly stronger.
Putting
$$
S(q) = \frac{q}{\phi(q)} \mathop{{\sum_{\chi}}^{\ast}} \max_{2 \leq y \leq x} \abs{\psi(y,\chi)},
$$
we have by partial summation that
\begin{align*}
\sum_{Q_1 < q \leq Q}\frac{1}{\phi(q)} \mathop{{\sum_{\chi}}^{\ast}} \max_{2 \leq y \leq x} \abs{\psi(y,\chi)} 
&= \frac{1}{Q} \sum_{q \leq Q} S(q) - \frac{1}{Q_1} \sum_{q \leq Q_1} S(q) + \int_{Q_1}^{Q} \left(\sum_{q \leq t} S(q)\right) \frac{dt}{t^2}. 
\end{align*}

Then, using \eq{MVT} to estimate the sums of the form $\sum S(q)$, the above is
\begin{align*}
\Comment{
&< \frac{4(c_1-1)}{5E_0} \rbr{\frac{x}{Q} + Q x^{\frac{1}{2}} + x^{\frac{5}{6}}} (\log x)^4 + \frac{c-1}{E_0} \int_{Q_1}^{Q} \rbr{ \frac{x}{t^2} + x^{1/2} + \frac{x^{\frac{5}{6}}}{t} } (\log x)^{\frac{7}{2}} dt \\
}
&\leq \frac{4(c_1-1)}{5E_0} \rbr{ 4\frac{x}{Q} + 2x^{\frac{1}{2}}Q + 6 x^{\frac{2}{3}} Q^{\frac{1}{2}} + 5x^{\frac{5}{6}} 
+ \int_{Q_1}^{Q} \rbr{ \frac{4x}{t^2} + 2x^{\frac{1}{2}} + \frac{6x^{\frac{2}{3}}}{t^{\frac{1}{2}}} + \frac{5x^{\frac{5}{6}}}{t} } dt   } (\log x)^{\frac{7}{2}} \\
&\leq \frac{4(c_1-1)}{5E_0} \rbr{ 4\frac{x}{Q_1} + 4x^{\frac{1}{2}}Q + 18 x^{\frac{2}{3}} Q^{\frac{1}{2}} + 5x^{\frac{5}{6}} \log \rbr{ \frac{eQ}{Q_1} } } (\log x)^{\frac{7}{2}}
\end{align*}
as required.
\end{proof}

\section{Proof of Corollary \ref{CEBVT}}

\begin{proof} 
Define
\begin{align*}
\pi_{1}(y) = \sum_{2 \leq n \leq y} \frac{\Lambda(n)}{\log n} = \sum_{p^k \leq y} \frac{1}{k}
\end{align*}
and
\begin{align*}
\pi_{1}(y;q,a) \quad = \mathop{\sum_{2 \leq n \leq y}}_{n \equiv a \text{ (mod $q$)}} \frac{\Lambda(n)}{\log n} \quad = \mathop{\sum_{p^k \leq y}}_{p^k \equiv a \text{ (mod $q$)}} \frac{1}{k}.
\end{align*}
We have
\begin{align*}
\pi_1(y;q,a) - \pi(y;q,a) 
= 
\sum_{2 \leq k \leq \frac{\log y}{\log 2}} \mathop{\sum_{p \leq y^{1 / k}}}_{p^k \equiv a \text{ (mod $q$)}} \frac{1}{k}
\leq
\sum_{2 \leq k \leq \frac{\log y}{\log 2}} \frac{\pi(y^{\frac{1}{2}})}{2} 
< 
2y^{\frac{1}{2}},
\end{align*}
where the last inequality follows from \eqref{pi upper bound}. 
Similarly,
$$
\pi_1(y) - \pi(y) < 2y^{\frac{1}{2}}.
$$
Moreover, by partial summation,
\begin{align*}
\abs{\pi_1(y;q,a) - \frac{\pi_1(y)}{\phi(q)}} &= \abs{\frac{\psi(y;q,a) - \psi(y) / \phi(q)}{\log y} -\int_{2}^{y} \frac{\psi(t;q,a) - \psi(t) / \phi(q)}{t \log^2 t} dt} \\
&\leq \frac{1}{\log 2} \abs{\psi(y;q,a) - \frac{\psi(y)}{\phi(q)}}+ \max_{2 \leq t \leq y} \abs{\psi(t;q,a) - \frac{\psi(t)}{\phi(q)}} \left(\frac{1}{\log{2}}-\frac{1}{\log{y}}\right).
\end{align*}
It then follows by the triangle inequality that the left hand side of (\ref{EBVTpi}) is
\begin{align*}
&\leq \frac{2}{\log 2} \sum_{\substack{q \leq Q\\ \ell(q)>Q_1}}  \max_{2 \leq y \leq x} \mathop{\max_{a}}_{(a,q)=1} \abs{ \psi(y,q,a) - \frac{\psi(y)}{\phi(q)} } 
+ 2x^{\frac{1}{2}} \sum_{\substack{q \leq Q \\ \ell(q)>Q_1}} \rbr{1+\frac{1}{\phi(q)}}.
\end{align*}
The first term is estimated by \eq{EBVTpsi}.  So to establish \eq{EBVTpi} we just note that, for $x \geq 4$, 
$$
2x^{\frac{1}{2}} \sum_{\substack{q \leq Q \\ \ell(q)>Q_1}} \rbr{1+\frac{1}{\phi(q)}}
$$ 
is less than the expression on the right of \eq{EBVTpsi} without $c_1$ present.  
\end{proof}

\end{document}